\documentclass[12pt,reqno]{amsart}
\usepackage{amssymb, amsfonts, amsbsy, latexsym, epsfig, color}
\usepackage{amsmath} 
\usepackage{amsthm}
\usepackage{amsmath,amssymb}
\newcommand{\vertiii}[1]{{\left\vert\kern-0.25ex\left\vert\kern-0.25ex\left\vert #1
    \right\vert\kern-0.25ex\right\vert\kern-0.25ex\right\vert}}
\DeclareMathOperator*{\argmin}{arg\,min}
\newtheorem{Thm}{Theorem}[section]

\newtheorem{lemma}[Thm]{Lemma}

\newtheorem{proposition}[Thm]{Proposition}
\newtheorem{definition}[Thm]{Definition}

\newtheorem{theorem}[Thm]{Theorem}

\newcommand\cC{{\mathcal{C}}}

\newcommand\cE{{\mathcal{E}}}

\newcommand\cH{{\mathcal{H}}}

\newcommand{\bitem}{\begin{itemize}}
\newcommand{\eitem}{\end{itemize}}
\newcommand{\benum}{\begin{enumerate}}
\newcommand{\eenum}{\end{enumerate}}
\newcommand{\beq}{\begin{equation}}
\newcommand{\eeq}{\end{equation}}
\newcommand{\ip}[2]{\langle#1,#2\rangle}

\newcommand{\spann}{\mbox{\rm span}}

  \newcommand{\R}{\mathbb{R}}
 \newcommand{\N}{\mathbb{N}}
 \newcommand{\C}{\mathbb{C}}
 \newcommand{\Z}{\mathbb{Z}}

\DeclareMathOperator*{\supp}{supp}

\def\ZZ{\mathbb{Z}}
\def\RR{\mathbb{R}}
\def\ZZ{\mathbb{Z}}

\def\cC{{\mathcal{C}}}

\def\cH{{\mathcal{H}}}

\def\cH{\mathcal{H}}

\newcommand{\gk}[1]{{\color{black}#1}}
\newcommand{\wq}[1]{{\color{black}#1}}

\begin{document}

\title{Optimal Compressive Imaging of Fourier Data}

\author[G. Kutyniok]{Gitta Kutyniok}
\address{Department of Mathematics, Technische Universit\"at Berlin, 10623 Berlin, Germany}
\email{kutyniok@math.tu-berlin.de}

\author[W.-Q Lim]{Wang-Q Lim}
\address{Fraunhofer Image \& Video Coding Group, Fraunhofer Institute for Telecommunications --
Heinrich Hertz Institute (Fraunhofer HHI), 10587 Berlin, Germany}
\email{wang.lim@hhi.fraunhofer.de}

\thanks{G.K. acknowledges support by the Einstein Foundation Berlin, by the Einstein Center for Mathematics Berlin
(ECMath), by Deutsche Forschungsgemeinschaft (DFG), by the DFG Collaborative Research Center TRR 109 ``Discretization
in Geometry and Dynamics'', by the DFG Research Center {\sc Matheon} ``Mathematics for key technologies'' in Berlin,
and by the European Commission. W.L. would like to thank the DFG Collaborative Research Center TRR 109 ``Discretization
in Geometry and Dynamics'' and the DFG Research Center {\sc Matheon} ``Mathematics for key technologies'' in Berlin
for its support.}

\begin{abstract}
Applications such as Magnetic Resonance Tomography acquire imaging data by point samples of their Fourier transform.
This raises the question of balancing the efficiency of the sampling strategies with the approximation accuracy of
an associated reconstruction procedure. In this paper, we introduce a novel sampling-reconstruction scheme based on
a random anisotropic sampling \gk{pattern} and a compressed sensing type reconstruction strategy with a variant of dualizable
shearlet frames as sparsifying representation system. For this scheme, we prove asymptotic optimality in an approximation
theoretic sense for cartoon-like functions as a model class for the imaging data. Finally, we present numerical experiments
showing the superiority of our scheme over other approaches.
\end{abstract}



\maketitle

\section{Introduction}\label{sec:intro}

In the age of Big Data, acquiring data is of tremendous importance, but a highly difficult task. One the
one hand, one aims for high subsampling of the original data in the sense of an efficient sensing process,
whereas on the other hand the reconstruction procedure should be both efficient and reconstruct the original
data with high accuracy. A breakthrough could recently be achieved by the introduction of the quite
general applicable methodology of compressed sensing in the two parallel papers \cite{CRT06} and \cite{Don06c},
which allows high subsampling rates using random measurement \gk{matrices} alongside efficient reconstruction
procedures such as $\ell_1$ minimization.

Each new technology requires a sensibly adapted methodology for data acquisition. Particular attention has
recently been paid to the general scenario of data acquisition from the Fourier transform of the original data.
For various applications such as Magnetic Resonance Imaging (MRI), Electron Microscopy (EM), Fourier Optics, Reflection
Seismology, and X-ray Computed Tomography (see also \cite{AHRT14}), the technology in fact only enables access
to the Fourier transform of the imaging data, allowing the acquisition of point samples. Although lately
various suitably adapted sampling-reconstruction schemes have been introduced and studied and some intriguing
results have been derived, the optimality of compressed sensing based schemes is still an open problem.

\subsection{General Approach to Sampling-Reconstruction Schemes for Fourier Data} \label{subsec:generalapproach}

For the aforementioned class of applications, the data acquisition process can be modelled as follows. Letting
$f \in \cH$, where $\cH = L^2(\R^2)$ or $\ell^2(\Z^2)$, presumably satisfy\gk{ing} additional regularity assumptions, and
given a sampling set $\Delta \subseteq \Z^2$, we acquire the point samples
\[
(\hat{f}(n))_{n \in \Delta} = (\ip{f}{e_n})_{n \in \Delta},
\]
where $e_n := e^{2\pi i \ip{\cdot}{n}}$, aiming to reconstruct $f$ from knowledge of those efficiently. A very general concept
which most approaches adopt requires the following ingredients:
\begin{itemize}
\item {\em Data Acquisition.} The sampling is performed by the map
\[
f \mapsto (\langle f,e_n \rangle)_{n \in \Delta},
\]
with $\Delta$ being \gk{coarsely speaking} as small as possible, giving the subsampling rate. We remark that applications typically put additional
constraints on the sampling set $\Delta$ such as sampling along lines or curves (for a first theoretical analysis of such
sampling scenarios, see \cite{GRUV15}). However, we will disregard such restrictions in order to derive a
benchmark result, which holds for a whole range of applications.
\item {\em Sparse Approximation.} Given a class of data, a widely accepted paradigm to date is the existence of a representation
system which provides sparse approximation of this class, which in the reconstruction process can be used for regularization of
the inverse problem. This requires to first indentify a model situation, $\cC \subseteq \cH$, say, and second an associated
representation system $(\psi_\lambda)_{\lambda \in \Lambda}$ -- which we for now assume to constitute an orthonormal basis --,
providing rapid decay of the error of the best $N$-term approximation of each element $f \in \cC$. Thus, we, in particular,
derive an expansion
\beq \label{eq:sparseapprox}
f = \sum_{\lambda \in \Lambda} c_\lambda \psi_\lambda
\eeq
with $(c_\lambda = \ip{f}{\psi_\lambda})_{\lambda \in \Lambda}$ being sparse.
\item {\em Reconstruction.} The reconstruction procedure then consists of solving the linear system of equations given by
\[
\left(\langle f,e_n \rangle = \sum_{\lambda \in \Lambda} \langle \psi_\lambda,e_n\rangle \tilde{c}_\lambda\right)_{n \in \Delta}
\mapsto (\hat{c}_\lambda)_{\lambda \in \Lambda},
\]
aiming for a small error $\|f - \sum_{\lambda \in \Lambda} \gk{\hat{c}_\lambda} \psi_\lambda\|$. This problem fits precisely into
the compressed sensing framework, which advocates to solve this problem as the convex optimization problem with sparsity
constraint
\beq \label{eq:l1general}
(\hat{c}_\lambda)_{\lambda \in \Lambda} = \mbox{argmin}_{(\tilde{c}_\lambda)_{\lambda \in \Lambda}}
\|(\tilde{c}_\lambda)_{\lambda \in \Lambda}\|_1 \; \mbox{s.t.} \; \Big(\langle f,e_n \rangle = \sum_{\lambda \in \Lambda}
\langle \psi_\lambda,e_n\rangle \tilde{c}_\lambda\Big)_{n \in \Delta},
\eeq
delivering results for the accuracy of reconstruction dependent on the sparse approximation property and the incoherence of
the measurement matrix $(\langle \psi_\lambda,e_n\rangle)_{\lambda \in \Lambda,n \in \Delta}$. It is well-known that in a
(finite) compressed sensing scenario $y = Ax$ with $x$ being sparsely approximated by a representation system, a random
$M \times N$ measurement matrix $A$ such as a Gaussian iid matrix allows for an optimally small number $M$ of measurement
vectors while still recovering $x$ exactly with high probability using $\ell_1$ minimization (for precise statements we
refer to \cite{DDEK12}). This already indicates the usefulness of random sampling sets $\Delta$.
\end{itemize}

\subsection{Sparse Sampling Results}

Recovery from Fourier samples via this general approach involving compressed sensing ideas has already been
extensively studied \gk{and} many practically, but also highly interesting theoretical \gk{results are} already available. The
maybe first work introducing compressed sensing to this sampling scenario in the setting of MRI is due to
Lustig, Donoho, Santos, and Pauly \cite{LDP07,LDSP08}, who used wavelets as sparsifying system and provided
very convincing empirical evidence of the superiority of this approach.

The first rigorous analysis, which might be considered as a theoretical breakthrough, of this scheme is due
to Krahmer and Ward \cite{KW13}. In their work, they study a finite dimensional setting, use Haar wavelets as
sparsifying system, and solve the inverse problem not by $\ell_1$ minimization, but by TV-regularization.
Intriguingly, by analyzing local coherence properties between wavelets and Fourier elements, they are able
to show that the variable density sampling, whose roots can be detected in \cite{LDP07,LDSP08}, is indeed
natural when aiming for a small incoherence of $(\langle \psi_\lambda,e_n\rangle)_{\lambda \in \Lambda,n \in \Delta}$.

A quite impressive body of work on Fourier sampling is due to Adcock and Hansen et al. In the continuum setting,
in \cite{AHP13,AHKM15} they are able to in fact prove an optimality result for a sampling-reconstruction scheme,
which is though not based on compressed sensing but generalized sampling \cite{AHRT14}. The sparsifying system\gk{s}
are in this case wavelet bases, first in $L^2(\R)$, then in $L^2(\R^2)$. This is however a deterministic strategy,
which is typically inferior to a random sampling-reconstruction scheme. The general theory developed in \cite{AHPR15}
can be identified as another fundamental contribution by these authors, in which they show the power of using the
multiscale structure of a sparsifying orthonormal basis for deriving superior recovery guarantees when considering
the infinite-dimensional setting in compressed sensing.

Finally, we wish to mention the paper \cite{SYSB14} by Shi, Yin, Sankaranarayanan, and Baraniuk, in which the
ideas from \cite{LDP07,LDSP08} are taken on a new level by considering the true 3D data of MRI and exploiting
the appearing joint sparsity of the wavelet coefficient to develop a scheme for dynamic MRI.

Certainly, an abundance of further, mostly empirical contributions does exist, and we refer to \cite{AHRT14}
for additional references.

\subsection{Frames versus Orthonormal Bases} \label{subsec:frames}

The knowledgable reader will have noticed that although in the previously described approaches the sparsifying
representation system $(\psi_\lambda)_{\lambda \in \Lambda}$ always formed an orthonormal basis, sparse
approximation results typically require the more general concept of frames. Indeed, frames provide non-unique
expansions due to their redundancy, thereby allows for much sparse representations and approximations.

A {\em frame} for $\cH$ is a sequence
$(\psi_\lambda)_{\lambda \in \Lambda} \gk{\subseteq \cH}$ satisfying $A \|f\|^2 \le \sum_{\lambda \in \Lambda} |\langle f ,
\psi_\lambda \rangle |^2 \le B \|f\|^2$ for all $f \in \cH$ with $0 < A \le B < \infty$. If  the frame
bounds $A$ and $B$ can chosen to be equal, it is typically called {\em tight frame}; in case of $A=B=1$,
a {\em Parseval frame}. Analysis of an element $f \in \cH$ by a frame $(\psi_\lambda)_{\lambda \in \Lambda}$
is achieved by application of the {\em analysis operator} $T$ given by
\[
T : \cH \to \ell^2(\Lambda), \quad f \mapsto (\langle f,\psi_\lambda \rangle)_{\lambda \in \Lambda}.
\]
Reconstruction of $f$ from the sequence of {\em frame coefficients} $(\langle f,\psi_\lambda \rangle)_{\lambda \in \Lambda}$
is possible by utilizing the adjoint operator $T^*$, since it can be shown that
\[
f = \sum_{\lambda \in \Lambda} \ip{f}{\psi_\lambda} (T^*T)^{-1}\psi_\lambda \quad \mbox{for all } f \in \cH.
\]
The system $((T^*T)^{-1}\psi_\lambda)_{\lambda \in \Lambda}$, required for this reconstruction formula, forms again
a frame, called the {\em canonical dual frame}. In general, a frame $(\tilde{\psi_\lambda})_{\lambda \in \Lambda}$
is typically referred to as {\em (alternate) dual frame}, provided that it satisfies
\[
f = \sum_{\lambda \in \Lambda} \ip{f}{\psi_\lambda} \tilde{\psi}_\lambda
= \sum_{\lambda \in \Lambda} \ip{f}{\tilde{\psi}_\lambda} \psi_\lambda \quad \mbox{for all } f \in \cH.
\]
For more details on frame theory, we refer to \cite{Chr03}.

If the sparsifying system $(\psi_\lambda)_{\lambda \in \Lambda}$ forms a frame, \gk{some} difficulties arise in the
general approach described in Subsection \ref{subsec:generalapproach} due to the fact that a dual frame is
often not accessible in closed form and its computation \gk{is} expensive and typically instable. Since it is
\gk{most of the time only} known that the sequence of frame coefficients
$(\langle f,\psi_\lambda \rangle)_{\lambda \in \Lambda}$ is sparse, instead of the expansion \eqref{eq:sparseapprox}
the expansion
\[
f = \sum_{\lambda \in \Lambda} c_\lambda \tilde{\psi}_\lambda
\]
\gk{has to be considered. Note that this is satisfied for $(c_\lambda = \ip{f}{\psi_\lambda})_{\lambda \in \Lambda}$.}
This \gk{viewpoint} is closely related to the novel framework of {\em co-sparsity} \cite{NDEG12}. This then leads to the
following replacement of \eqref{eq:l1general}:
\[
(\hat{c}_\lambda)_{\lambda \in \Lambda} = \mbox{argmin}_{(\tilde{c}_\lambda)_{\lambda \in \Lambda}}
\|(\tilde{c}_\lambda)_{\lambda \in \Lambda}\|_1 \; \mbox{s.t.} \; \Big(\langle f,e_n \rangle = \sum_{\lambda \in \Lambda}
\langle \tilde{\psi}_\lambda,e_n\rangle \tilde{c}_\lambda\Big)_{n \in \Delta},
\]

\subsection{Notion of Optimality} \label{subsec:opt}

Aiming for optimality of a sampling-reconstruction scheme in an approximation theoretic sense, we will now
assume a continuum model $\mathcal{C} \subseteq L^2(\RR^2)$. Since signals \gk{in} nature \gk{are} almost always of
continuous type \gk{or in other words} we live in a continuous world, this choice seems appropriate.
The {\em optimal best $N$-term approximation rate} is then defined as $O(N^{-\alpha})$ with $\alpha > 0$
maximal such that
\[
\inf_{f_N = \sum_{\lambda \in \Lambda_N} c_\lambda \psi_\lambda, \# \Lambda_N = N}\|f - f_N\|_2 \lesssim N^{-\alpha} \mbox{ as } N \to \infty \mbox{ for all } f \in \cC,
\]
and for any frame $(\psi_\lambda)_{\lambda \in \Lambda}$ for $L^2(\RR^2)$.

Let $(\Delta_M)_M \subseteq \Z^2$ with $\#\Delta_M = M$ and $M \to \infty$ now be a sequence of growing
sampling sets, and let
\[
\mathcal{R}: \mathcal{C} \times \Delta \to L^2(\RR^2), \quad \Delta := \bigcup_M \{\Delta_M\}
\]
be a reconstruction procedure. Then we call a sampling-reconstruction scheme $(\mathcal{C}, \Delta, \mathcal{R})$
{\em asymptotically optimal}, if, for all $f \in \cC$,
\[
\|f - \mathcal{R}(f,\Delta_M)\|_2 \lesssim M^{-\alpha} \mbox{ as } M \to \infty.
\]

\subsection{Anisotropic Model and Shearlet Frames}

In multivariate situations, anisotropic features such as edges in images or shock fronts in the solution of
certain classes of partial differential equations are governing features. In fact, this becomes quite apparent
if considering the data \gk{which} applications such as MRI or EM face. Another interesting aspect is that in fact
neurophysiologists have strong evidence that the neurons in the visual cortex of humans also react very
strongly to those features \wq{\cite{FHH}.} A very classical model for this is the class
of cartoon-like functions, consisting of $L^2$-functions which are $C^2$ apart from a $C^2$ discontinuity
curve, i.e., in particular, piecewise smooth. It was proven by Donoho in \cite{Don01} that the
optimal best $N$-term approximation rate for this model class is $\alpha = 1$.

Assuming this model situation, maybe the most widely used representation system is the multiscale system of shearlets originally
introduced in \cite{GKL06} (see also the survey paper \cite{KL12}). These form a directional representation
system which delivers the optimal approximation rate up to a $\log$-factor \cite{KL11}, but which -- in
contrast to the previously advocated system of curvelets \cite{CD04} -- also allow a faithful implementation
by a unified treatment of the continuum and digital realm \cite{KLR14} and provides compactly supported
variants for high spatial localization \cite{KKL12}. In the compactly supported version, \gk{a} shearlet \gk{system forms}
a non-tight frame requiring knowledge of the dual frame for reconstruction.

Very recently, in addition to the ``\gk{standard} shearlet systems'', a novel type of shearlet systems has been
introduced in \cite{KL15} coined {\em dualizable shearlets}, which satisfy the following key features:
\bitem
\item[(DS1)] Dualizable shearlet systems are composed of orthonormal bases.
\item[(DS2)] These systems have compactly supported elements.
\item[(DS3)] Dualizable shearlet frames possess one dual frame for which a closed formula exists.
\item[(DS4)] They provide optimally sparse approximations of \gk{cartoon-like functions}.
\eitem
This construction will be one of the backbones of the sampling-reconstruction scheme we will introduce
in this paper.

\subsection{Our Contribution}

The main goal in this paper is to introduce a provably optimal sampling-reconstruction scheme for
reconstruction from Fourier measurements. For this, we follow the very general approach outlined in
Subsection \ref{subsec:generalapproach}. We consider with the class of cartoon-like functions a
{\em continuum model} to not obscure continuum elements of geometry and to be consistent with the fact
that we live in a continuous world. In fact, for this data acquisition problem, this is the first time
that a specifically adapted model has been analyzed.

As a sparsifying representation system, we introduce and analyze a {\em variant of dualizable shearlets},
which form\gk{s} a {\em frame}, hence requiring the adaption of the general framework of sampling-reconstruction
schemes to frames \gk{as discussed in} Subsection \ref{subsec:frames}. The redundancy of the reconstruction system
causes significant problems when analyzing their performance within a sampling-reconstruction scheme,
which we tackle by carefully exploiting the inner structure of dualizable shearlets as a union of orthonormal
bases. We expect that such an approach could be useful in general to extend methodologies requiring an
orthnormal basis as sparsifying system to the frame setting.

The sampling sets $\Delta_M$ -- with the asymptotics linked to specific scales -- are designed to
provide maximal incoherence of the measurement matrices given by $A:=(\langle \psi_\lambda,e_n\rangle)_{\lambda \in \Lambda,n \in \Delta}$
measured by the isometry constant $\delta_k$ of the restricted isometry property \cite{DDEK12}
\[
(1-\delta_k) \|x\|_2^2 \le \|A_Sx\|_2^2 \le (1+\delta_k) \|x\|_2^2 \quad \mbox{for all } x,
\]
where $S$ is an index set of size at most $k$ and $A_S$ denotes the selection of columns indexed by $S$.
In fact, we measure it more carefully by considering the local incoherence introduced in \cite{KW13}.
In contrast to the (isotropic) variable density sampling scheme introduced and analyzed in \cite{KW13},
our construction leads not unexpectedly to a set of {\em highly directional random sampling schemes} $\Delta_M$.
It is though interesting to notice in which way the directionality appears, strongly concentrated
along the angle bisectors.

We then prove that \gk{our} novel sampling-reconstruction scheme is indeed asymptotically optimal
(Theorem \ref{thm:main}) in the sense of Subsection \ref{subsec:opt}, which we will make a bit more precise
later \gk{(Definition \ref{defi:optimal})}. To our knowledge this is the first {\em optimality result for sampling-reconstruction schemes} for
Fourier data based on compressed sensing in the continuum setting. Besides the previously mentioned ingredients,
\gk{one further key} idea of our analysis is the approximation of the continuum situation by finite dimensional \gk{objects}
--  in particular, at each level we consider a finite-dimensional shearlet system for reconstruction -- in the
spirit of, yet conceptionally different from the finite section method like approach of generalized sampling
\cite{AHRT14}. Moreover, our asymptotic analysis relies heavily on the multiscale structure of shearlets,
though in contrast to the multiscale sparsity studied in \cite{AHPR15}, we \wq{also investigate sparsity patterns
for cartoon-like functions with respect to directional parameters in the shearlet representation.}

\subsection{Outline}

The paper is organized as follows. We start by introducing the main definitions and results for dualizable shearlets
in Section \ref{sec:dualizable}. Section \ref{sec:scheme} is the\gk{n} devoted to introduce and discuss our novel
sampling-reconstruction scheme, and state our optimality result. Its proof is presented in Section \ref{sec:proof}.
Finally, numerical experiments comparing our scheme with other schemes are discussed in Section \ref{sec:numerics}.

\section{Dualizable Shearlets}\label{sec:dualizable}

Dualizable shearlets are one main ingredient in our sampling-reconstruction scheme, as discussed in
the introduction. Hence this section shall serve as a review of the definition and main properties
of dualizable shearlet frames. For further details we refer to \cite{KL15}.

As already indicated in the introduction, the composition of features (DS1)--(DS4) -- which also distinguishes dualizable
shearlet systems from ``\gk{standard}'' shearlet systems (see \cite{KL12}) -- will enable us to derive an optimal
sampling-reconstruction scheme.

\subsection{Definition and Basic Notions}

The construction consists of two steps. First, so-called shearlet-type wavelet systems are defined, which
are unions of orthonormal bases, indexed by a shearing parameter. Second, a particular set of filters is
used to cut out certain parts of those systems, to ensure directionality. Convolving both then leads
to dualizable shearlet systems. In the spirit of the cone-based definition of \gk{standard} shearlet systems
(cf. \cite{KL12}),
these systems are composed of two parts: a system consisting of functions, for which the essential supports of
their Fourier transforms are horizontally aligned, and a similar system associated with a vertical cone in
Fourier domain -- which is typically the rotated first system. We wish to mention that concerning the
dualizable shearlet system we will only present the construction for the horizontal Fourier cone in detail,
the vertically aligned system will be derived by switching the variables.

As our first step, aiming to construct a family of orthonormal bases for $L^2(\R^2)$ for each shearing direction,
\gk{we} choose two univariate functions $\varphi_1$ and $\psi_1$ to satisfy the following definition. We wish to mention
that \gk{as in the sequel} the symbol $\lesssim$ indicates the asymptotics for $J \to \infty$.

\begin{definition}
Let $\varphi_1,  \psi_1 \in L^2(\R)$ be compactly supported and satisfy the support condition
\[
\delta_{\varphi_1} = \inf_{\xi \in [-\frac12,\frac12]} |\hat\varphi_1(\xi)| > 0,
\]
as well as, for some $\rho \in (0,\frac{2}{13})$, $\alpha \ge \frac{6}{\rho}+1$, and $\beta > \alpha + 1$, the decay conditions
\[
\Bigl|\Bigl(\frac{d}{d\xi}\Bigr)^{\ell}\hat\psi_1(\xi)\Bigr| \lesssim \frac{\min\{1,|\xi|^{\alpha}\}}{(1+|\xi|)^{\beta}}
\quad \mbox{and} \quad
\Bigl|\Bigl(\frac{d}{d\xi}\Bigr)^{\ell}\hat\varphi_1(\xi)\Bigr| \lesssim \frac{1}{(1+|\xi|)^{\beta}}
\quad \mbox{for } \ell=0,1.
\]
We further assume that the system
\[
\{\varphi_1(\cdot-m) : m \in \Z\} \cup \{2^{j/2}\psi_1(2^j\cdot-m) : j \ge 0, m \in \Z\}
\]
forms an orthonormal basis for $L^2(\R)$. Then we refer to the pair $(\varphi_1,  \psi_1)$ as \emph{$(\rho, \alpha, \beta)$-admissible}.
\end{definition}

We remark that the condition to constitute an orthonormal basis can indeed be fulfilled by \cite{Dau92}.

Let now $(\varphi_1,  \psi_1)$ be $(\rho, \alpha, \beta)$-admissible. Based on these univariate functions, for each
$x = (x_1,x_2) \in \R^2$, we define
\beq \label{eq:psi_0_p}
\psi^0(x) := \psi_1(x_1)\varphi_1(x_2) \quad \mbox{and} \quad
\psi^{p}(x) := 2^{(p-1)/2}\psi_1(x_1)\psi_1(2^{p-1}x_2) \; \mbox{for } p > 0
\eeq
\gk{as well as}
\beq \label{eq:varphi_0p}
\varphi^0(x) := \varphi_1(x_1)\varphi_1(x_2) \quad \mbox{and} \quad
\varphi^{p}(x) := 2^{(p-1)/2}\varphi_1(x_1)\psi_1(2^{p-1}x_2) \; \mbox{for } p > 0.
\eeq
In this construction, the parameter $p$ enables a dyadic substructure in vertical direction.
For a fixed integer $j_0 \ge 0$, we now consider the system given by
\[
\{|\mathrm{det}(A_{j_0})|^{1/2}\varphi^p(A_{j_0} \cdot-D_pm), \:
|\mathrm{det}(A_j)|^{1/2}\psi^p(A_j \cdot-D_pm) : j \ge j_0, p \ge 0\},
\]
where $D_p = \mathrm{diag}(1,2^{-\max\{p-1,0\}})$. By \cite{KL15}, this system forms an orthonormal
basis for $L^2(\R^2)$.

We next recall that the shear parameter for a \gk{standard} shearlet system should equal
$\frac{k}{2^{\lceil j/2 \rceil}}$ for $|k| \leq 2^{\lceil j/2\rceil}$. Aiming to parameterize those quotients
we define the injective map
\[
s : \{(j,q) : j = 0, q = 0\} \cup \{j : j \ge 0\} \times \{q : |q| \leq 2^{j}, \, q \in 2 \Z +1\} \to [-1,1], \,\, s(j,q) := \frac{q}{2^{j}},
\]
as well as
\[
\mathbb{S} = \{s(j,q) = 0 : j = 0, q = 0\} \cup \{s(j,q) : j \ge 0,\, |q| \leq 2^{j}, \, q \in 2 \Z +1\}.
\]
These preparations enable us to now defined shearlet-type wavelet systems.

\begin{definition} \label{def:shearlettypewavelets}
Let $\rho \in (0,\wq{\frac{1}{12}})$, $\alpha \ge \frac{6}{\rho}+1$, and $\beta > \alpha + 1$.
Let $\varphi_1, \psi_1 \in L^2(\R)$ with $(\varphi_1, \psi_1)$ being $(\rho, \alpha, \beta)$-admissible,
and let $\varphi^p, \psi^p \in L^2(\R^2)$, $p \ge 0$ satisfy \eqref{eq:psi_0_p} and \eqref{eq:varphi_0p}.
Further, set $D_p := \mathrm{diag}(1,2^{-d_p})$ with $d_p := \max\{p-1,0\}$. Then, for each shear parameter
$s := s(\lceil j_0/2\rceil,q_0) \in \mathbb{S}$, where $j_0$ is the smallest nonnegative integer such that
$s = \frac{q_0}{2^{\lceil j_0/2\rceil}}$, we define the {\em shearlet-type wavelet system} $\Psi_s(\varphi_1, \psi_1)$ by
\[
\Psi_s(\varphi_1, \psi_1) := \{\varphi_{j_0,s,m,p}, \psi_{j,s,m,p} : j \ge j_0, m \in \Z^2, p \ge 0\},
\]
where
\[
\varphi_{j,s,m,p} := |\mathrm{det}(A_j)|^{1/2}\varphi^p(A_jS_s\cdot-D_p m) \;\: \mbox{ and } \;\: \psi_{j,s,m,p} := |\mathrm{det}(A_j)|^{1/2}\psi^p(A_jS_s\cdot-D_p m).
\]
\end{definition}

The second step now consists in defining a filtering procedure. To be able to easier refer to properties
of a generator and an associated filter, we introduce the following notion.

\begin{definition} \label{defi:g}
Let $\rho \in (0,\frac{2}{13})$, $\alpha \ge \frac{6}{\rho}+1$, and $\beta > \alpha + 1$, and
let $g \in L^2(\R^2)$ be a compactly supported function satisfying the conic support condition
\[
\delta_g = \inf_{\xi \in \Omega_{g}} |\hat g(\xi)| > 0, \quad \mbox{where } \Omega_{g} = \{\xi \in \R^2 : |\tfrac{\xi_2}{\xi_1}| < 1, \tfrac12 < |\xi_1| < 1\Big\},
\]
as well as the decay condition
\[
\Bigl|\Bigl(\frac{\partial}{\partial \xi_2}\Bigr)^{\ell}\hat g(\xi)\Bigr| \lesssim \frac{\min\{1,|\xi_1|^{\alpha}\}}{(1+|\xi_1|)^{\beta}(1+|\xi_2|)^{\beta}}
\quad \mbox{for } \ell=0,1.
\]
We then refer to the family of filters $G_s$, $s = s(\lceil j_0 /2\rceil,q_0) \in \mathbb{S}$ defined by
\begin{equation}\label{eq:dfilters}
\hat G_{0}(\xi) = |\hat \varphi^0(\xi)|^2+\sum_{j = 0}^{\infty} |\hat g(A^{-1}_j\xi)|^2
\quad \mbox{and} \quad
\hat G_{s}(\xi) = \sum_{j = j_0}^{\infty} |\hat g(A^{-1}_jS^{-T}_s \xi)|^2 \quad \mbox{for } s \neq 0
\end{equation}
as a \emph{family of $(\rho, \alpha, \beta; g)$-filters}.
\end{definition}

Having introduced the two necessary ingredients, we can now formally define dualizable shearlet systems as follows.
Notice that this definition will indeed require the systems derived by switching the variable in $\Psi_s(\varphi_1, \psi_1)$,
i.e., \wq{by $R = \left(
  \begin{array}{cc}
    0 & 1 \\
    1 & 0 \\
  \end{array}
\right)$.}

\begin{definition}[\cite{KL15}] \label{def:dualizableshearlets}
Let \wq{$\rho \in (0,\frac{1}{12})$}, $\alpha \ge \frac{6}{\rho}+1$, and $\beta > \alpha + 1$.
For any $s \in \mathbb{S}$, let $\Psi_s(\varphi_1, \psi_1)$ be a shearlet-type wavelet system, and let $(G_s)_{s \in \mathbb{S}}$
be a family of $(\rho, \alpha, \beta; g)$-filters. Then
the {\em dualizable shearlet system} $\mathcal{SH}(\varphi_1, \psi_1; g)$ is defined by
\[
\mathcal{SH}(\varphi_1, \psi_1; g) = \{\psi^{\ell}_{\lambda} : \lambda \in \Lambda_s, s \in \mathbb{S}, \ell = 0,1\}
\]
with index set
\[
\Lambda_{s} = \{(j,s,m,p) : j \in \{-1\} \cup \{j_0, j_0+1, \ldots\}, m \in \Z^2, p \ge 0\} \quad \mbox{for } s =  s(\lceil j_0 /2\rceil,q_0),
\]
where
\[
\psi^{0}_{\lambda} =
\left\{
\begin{array}{rcl}
G_s * \varphi_{j_0,s,m,p} & : & \lambda = (-1,s,m,p) \in \Lambda_s,\\
G_s * \psi_{j,s,m,p} & : &  \lambda = (j,s,m,p) \in \Lambda_s,
\end{array}
\right.
\]
and $\psi^{1}_{\lambda} = \psi^{0}_{\lambda} \circ R$.
\end{definition}

The following first two results from \cite{KL15} state the announced properties of orthonormality (DS1) and compact support (DS2).

\begin{proposition}[\cite{KL15}]\label{lemm:ONB}
For each $s \in \mathbb{S}$, the shearlet-type wavelet system $\Psi_s(\varphi_1, \psi_1)$ is an orthonormal basis for $L^2(\R^2)$.
\end{proposition}

\begin{proposition}[\cite{KL15}]\label{lem:compact}
Each dualizable shearlet system is compactly supported.
\end{proposition}

The third already mentioned key property, i.e., (DS3), is the explicit form of one of the associated dual frames, made precise in
the following statement.

\begin{theorem}[\cite{KL15}]\label{theo:dualframe}
Let \wq{$\rho \in (0,\frac{1}{12})$}, $\alpha \ge \frac{6}{\rho}+1$, and $\beta > \alpha + 1$. Let $\varphi_1, \psi_1 \in L^2(\R)$
with $(\varphi_1, \psi_1)$ being $(\rho, \alpha, \beta)$-admissible, and let $(G_s)_{s \in \mathbb{S}}$ be a family of
$(\rho, \alpha, \beta; g)$-filters. Further, let $\mathcal{SH}(\varphi_1, \psi_1; g) = \{\psi^{\ell}_{\lambda} : \lambda
\in \Lambda_s, s \in \mathbb{S}, \ell = 0,1\}$ be a dualizable shearlet system, which constitutes a frame for $L^2(\R^2)$.
Then
\[
\widetilde{\mathcal{SH}}(\varphi_1, \psi_1; g) = \{\tilde{\psi}^{\ell}_{\lambda} : \lambda \in \Lambda_s, s \in \mathbb{S}, \ell = 0,1\}
\]
is a dual frame for $\mathcal{SH}(\varphi_1, \psi_1; g)$, where, for $\lambda \in \Lambda_s$,
\[
\hat{\tilde{\psi}}^0_{\lambda} = \frac{\hat \psi^0_{\lambda}}{\sum_{s' \in \mathbb{S}}|\hat{G}_{s'}|^2+|\hat{G}_{s'} \circ \wq{R}|^2}
\quad \mbox{and} \quad
\tilde{\psi}^1_{\lambda} = \tilde{\psi}^0_{\lambda} \circ \wq{R}.
\]
\end{theorem}

\subsection{Sparse Approximation}

We will next discuss the optimal sparse approximation properties dualizable shearlet systems satisfy. This anticipate\gk{d}
behavior was before labeled as (DS4).

Since it is typically assumed that images are governed by anisotropic features, in particular, edges, a common model
is the so-called class of cartoon-like functions. We start by formally introducing this class, which was first defined
in \cite{Don01}.

\begin{definition} \label{defi:cartoon}
The set of \emph{cartoon-like functions} $\cE^2(\RR^2)$ is defined by
\[
\cE^2(\RR^2) = \{f \in L^2(\RR^2) : f = f_0 + f_1 \cdot \chi_{B}\},
\]
where $B \subset [0,1]^2$ is a nonempty, simply connected set with $C^2$-boundary, $\partial B$ has bounded curvature,
and $f_i \in C^2(\RR^2)$ satisfies $\mathrm{supp}  f_i \subseteq [0,1]^2$ and $\|f_i\|_{C^2} \le 1$ for $i=0, 1$.
\end{definition}

In \cite{Don01}, Donoho proved a benchmark result in the sense of a lower bound for the achievable decay
rate of the error of best $N$-term approximation. One should mention that he in fact showed a more general result
than the following statement; however this is what we require for our endeavour.

\begin{theorem}[\cite{Don01}]\label{theo:optimal}
Let $(h_i)_{i \in I} \subseteq L^2(\R^2)$ be a frame for $L^2(\R^2)$. Then, for any $f \in \cE^2(\RR^2)$, the $L^2$-error of
best $N$-term approximation by $f_N$ with respect to $(h_i)_{i \in I}$ satisfies
\[
\|f-f_N\|_2 \gtrsim N^{-1}  \qquad \text{as } N \rightarrow \infty.
\]
\end{theorem}

Let \gk{us} next assume that we are given a dualizable shearlet system
\[
\mathcal{SH}(\varphi_1, \psi_1; g) = \{\psi^{\ell}_{\lambda} : \lambda \in \Lambda_s, s \in \mathbb{S}, \ell = 0,1\}
\]
with associated dual frame
\[
\widetilde{\mathcal{SH}}(\varphi_1, \psi_1; g) = \{\tilde{\psi}^{\ell}_{\lambda} : \lambda \in \Lambda_s, s \in \mathbb{S}, \ell = 0,1\}
\]
as defined in Theorem \ref{theo:dualframe}. Utilizing the more compact notation
\[
\Lambda := \{0,1\} \times \bigcup_{s \in \mathbb{S}} \Lambda_s
\]
for the index sets of those systems, we now consider $N$-term approximations of $f \in \cE^2(\RR^2)$ of the form
\[
f_N = \sum_{(\ell,\lambda) \in \Lambda_N} \langle f,\psi^{\ell}_{\lambda}\rangle \tilde{\psi}^{\ell}_{\lambda},
\]
where $\Lambda_N \subseteq \Lambda$, $\# \Lambda_N = N$. It might be surprising to consider expansions in terms
of the dual frame, but in our case we have -- and would like to use -- knowledge on the decay of the frame
coefficients $(\langle f,\psi^{\ell}_{\lambda}\rangle)_{\ell, \lambda}$.

The following result from \cite{KL15} shows that the approximation rate of dualizable shearlets for cartoon-like functions
can be arbitrarily close to the optimal rate as the smoothness of the generators is increased, i.e., as $\rho \to 0$.

\begin{theorem}[\cite{KL15}]\label{thm:sparsity}
Let \wq{$\rho \in (0,\frac{1}{12})$}, $\alpha \ge \frac{6}{\rho}+1$, and $\beta > \alpha + 1$. Let $\varphi_1, \psi_1 \in L^2(\R)$
with $(\varphi_1, \psi_1)$ being $(\rho, \alpha, \beta)$-admissible, and let $(G_s)_{s \in \mathbb{S}}$ be a family of
$(\rho, \alpha, \beta; g)$-filters. Further, let $\mathcal{SH}(\varphi_1, \psi_1; g) = \{\psi^{\ell}_{\lambda} : \lambda
\in \Lambda_s, s \in \mathbb{S}, \ell = 0,1\}$ be a dualizable shearlet system, which constitutes a frame for $L^2(\R^2)$\gk{,
and } let $f \in \mathcal{E}^2(\R^2)$. Then
\[
\|f - f_N\|_2 \lesssim \wq{N^{-1+12\rho}} \cdot \log(N) \qquad \text{as } N \rightarrow \infty,
\]
where $f_N = \sum_{(\ell,\lambda) \in \Lambda_N} \langle f,\psi^{\ell}_{\lambda}\rangle \tilde{\psi}^{\ell}_{\lambda}$ with
$\Lambda_N \subseteq \Lambda$, $\# \Lambda_N = N$ is the $N$-term approximation
using the $N$ largest coefficients $(\langle f,\psi^{\ell}_{\lambda}\rangle)_{\ell,\lambda}$.
\end{theorem}

Thus, dualizable shearlets do indeed satisfy (DS1)--(DS4).

\section{Optimal Sampling-Reconstruction Scheme}\label{sec:scheme}

We now turn to introducing our sampling-reconstruction scheme and stating the associated optimality result.
But before, we first give a precise notion of optimality of sampling-reconstruction schemes, and -- due to
the technical difficulty of the construction and arguments -- we will also provide some intuition on the
construction of \gk{our} scheme.

\subsection{A Notion of Optimality}

In Subsection \ref{subsec:opt}, we already introduced a very general notion of optimality for general sampling-reconstruction
schemes. We now make this precise in the setting we are considering.

We \gk{choose} our model to be the model of cartoon-like functions $\cE^2(\RR^2)$ as defined in Definition \ref{defi:cartoon})
which, by Theorem \ref{theo:optimal}, admits the following optimal decay rate of the best $N$-term approximation:
\[
\|f - f_N\|_2 \lesssim N^{-1} \mbox{ as } N \to \infty \mbox{ for all } f \in \cE^2(\RR^2),
\]
where $f_N = \sum_{\lambda \in \Lambda_N} c_\lambda \psi_\lambda$ for some frame $(\psi_\lambda)_{\lambda \in \Lambda} \subseteq L^2(\RR^2)$.
Let now $J$ be a positive integer, and assume that we already constructed a set of sampling schemes
\beq \label{eq:sampling}
\Delta_J \subseteq \Z^2, J  > 0, \quad \mbox{with } J \mapsto \#\Delta_J \mbox{ strictly increasing}
\eeq
as well as an associated reconstruction scheme
\beq \label{eq:reconstruction}
\mathcal{R} = \mathcal{R}(\cE^2(\RR^2),\Delta) : \cE^2(\RR^2) \times \Delta \to L^2(\RR^2),
\eeq
where
\beq \label{eq:Delta}
\Delta = \bigcup_{J > 0} \gk{\{\Delta_J\}}.
\eeq

We then introduce the following notion of optimality. We emphasize that we take an asymptotic viewpoint concerning
optimality with $J$ being the parameter which provides the asymptotics. This is consistent with our continuum
model in $L^2(\R^2)$ in the sense that we aim for an optimal behavior in this limit situation.

\begin{definition} \label{defi:optimal}
A sampling-reconstruction scheme $\mathcal{R} = \mathcal{R}(\cE^2(\RR^2),\Delta)$ is called \emph{asymptotically optimal}, if,
for all $f \in \cE^2(\RR^2)$,
\[
\|f - \mathcal{R}(f,\Delta_J)\|_2 \lesssim (\#\Delta_J)^{-1} \quad \mbox{as } J \to \infty.
\]
\end{definition}

\subsection{Intuition of Our Scheme}\label{subsec:intuition}

Before defining the sampling sets $\Delta_J$, $J > 0$, and the reconstruction system similarly dependent on the scale $J$,
we will provide some first intuition for their choice. As already discussed in the introduction, we will choose a particular
version of dualizable shearlet systems. In fact, for each limiting scale $J$, we will choose a finite-dimensional version,
which will turn out to be sufficient to derive an asymptotically optimal sampling-reconstruction scheme. A key ingredient
will be the fact that those reconstruction systems will form orthonormal bases for each shear $s$.
The associated subset of the shearing parameters $\mathbb{S}$ associated with those finite-dimensional (dualizable) shearlet
systems, will be chosen to be
\[
\mathbb{S}_{J/2} := \{s(j,q) = 0 : j = 0, q = 0\} \cup \{s(\lceil j/2\rceil,q) : 0 \leq j \leq J,\, |q| \leq
2^{\lceil j/2 \rceil}, \, q \in 2 \Z +1\}, \; J > 0.
\]

The sampling sets have then to be constructed such that the associated Fourier elements are maximally incoherent with the
reconstruction system. \gk{To choose the sampling points in a suitable way requires analysis of the sparsity pattern of
shearlet expansions of cartoon-like functions.} Two sparsity pattern can be observed \gk{and need to be carefully
handled}:
\bitem
\item {\em Scale Dependence.} The first sparsity pattern is the distribution of sparsity with respect to the scales $j$.
For this, we drive a strategy which is inspired by the variable density sampling strategy from \cite{KW13}, analyzing the
local coherence terms of the shearlet elements and the elements of the Fourier basis. For each $s \in \mathbb{S}_{J/2}$,
a suitable probability density function will then lead to a measurement matrix satisfying the Restricted Isometry Property
with a sufficiently small restricted isometry constant with high probability. We remark that the orthogonality plays a
prominent role in this analysis. \gk{Then}, for each $s$, we will reconstruct via $\ell_1$ minimization, and finally patch the
solutions together.
\item {\em Shear Dependence.} The second sparsity pattern we encounter is the distribution of sparsity with respect to the
shear direction $s$. To keep this under control, we will later slightly restrict the class of cartoon-like functions to
ensure that the significant shearlet coefficients are almost equally distributed with respect to $s$, which in a sense
allows the patching of the solutions after the separate $\ell_1$ minimizations.
\eitem
As can be seen from this discussion, each $\Delta_J$ need\gk{s} to have a special substructure, namely, dependent on
the shear $s \in \mathbb{S}_{J/2}$. In fact, the sampling sets will be constructed as a union of sampling sets
$\Delta_{J,s}$.

\subsection{Sampling Systems}\label{subsec:samplingsystem}

We first define the sampling sets -- and with this the sampling systems --  which we choose for $(\Delta_J)_J$ in
\eqref{eq:sampling}. Letting $J > 0$, the {\em complete} sampling set has the form
\[
\Omega_{J} = \{(n_1,n_2) \in \Z^2 : \max(|n_1|,|n_2|) \lesssim 2^{J(1+\rho)}\}
\]
with associated Fourier basis
\[
\Phi_J = \{e_n := e^{2\pi i \ip{\cdot}{n}} : n = (n_1,n_2) \in \Omega_J\}.
\]

Let now $j_0$ be an arbitrarily fixed scale. We \gk{then} choose a sampling set $\Delta_{J,s}$ for each
$s = s(\lceil j_0/2 \rceil,q_0) \in \mathbb{S}_{J/2}$ of size
\[
m_{J,s} \sim 2^{\frac{J-j_0}{2}}\wq{2^{3J\rho}},
\]
i.e., \wq{dependent of} $s$.
The $m_{J,s}$ sampling points $\Delta_{J,s}$ in $\Omega_J$ are then chosen randomly
according to the probability density function $p_{J,s}$ defined by
\begin{equation}\label{eq:probability}
p_{J,s}(n) = \frac{c_s}{J^2(1+|n_1|)(1+|n_2-sn_1|)}
\end{equation}
over the discrete frequency domain $\Omega_J$, where the constant $c_s$ has obviously to be chosen so that
$\sum_{n \in \Omega_J} p_s(n) = 1$. Note that this condition implies that there exist some positive constants
$C_1$ and $C_2$ such that $C_1 < c_s < C_2$ for all $s \in \mathbb{S}_{J/2}$.

Thus, the random sampling sets $(\Delta_J)_J$ employed in our sampling-reconstruction scheme are chosen as
\beq \label{eq:samplingsets}
\Delta_J := \bigcup_{s \in \mathbb{S}_{J/2}} \Delta_{J,s}
\eeq
with associated sampling systems -- as subsystems of $\Phi_J$ --  given by
\[
\{e_{n} : n \in \Delta_{J,s}, s \in \mathbb{S}_{J/2}\}.
\]

Concerning the number of elements in each sampling set $\Delta_J$, we obtain the following result.

\begin{lemma} \label{lemm:deltaJ}
Let the sampling sets $(\Delta_J)_J$ be defined as in \eqref{eq:samplingsets}. Then
\[
\#\Delta_J \lesssim J \cdot 2^{J/2(1+\wq{6\rho})}.
\]
\end{lemma}

\begin{proof}
We have
\[
\#\Delta_J = \sum_{s \in \mathbb{S}_{J/2}} \wq{m_{J,s}} \lesssim \sum_{j_0 = 0}^{J} 2^{j_0/2}2^{\frac{J-j_0}{2}}\wq{2^{3J\rho}} \lesssim J \cdot 2^{J/2(1+\wq{6\rho})}.
\]
This proves the claim.
\end{proof}

\subsection{Reconstruction Systems}

We continue by describing the representation system\gk{s} we will employ for the reconstruction step.
For this, we let $\rho \in (0,\wq{\frac{1}{12}})$, $\alpha \ge \frac{6}{\rho}+1$, and $\beta > \alpha + 1$. Further, we
choose $\varphi_1, \psi_1 \in L^2(\R)$ with $(\varphi_1, \psi_1)$ being $(\rho, \alpha, \beta)$-admissible, and
$(G_s)_{s \in \mathbb{S}}$ to form a family of $(\rho, \alpha, \beta; g)$-filters. Let then
$\mathcal{SH}(\varphi_1, \psi_1; g) = \{\psi^{\ell}_{\lambda} : \lambda \in \Lambda_s, s \in \mathbb{S}, \ell = 0,1\}$
be the associated dualizable shearlet system.

We next define finite-dimensional shearlet systems depending on $J$ and $s$, which will serve as sparsifying systems
for any fixed $J$. For this, for each $s = s(\lceil j_0/2 \rceil,q_0) \in \mathbb{S}_{J/2}$, we define the index set
associated with those shearlets whose support \gk{basically} intersects the support of the cartoon-like function $f$, i.e.,
\beq \label{eq:L_sJ}
\hspace*{-0.2cm}\tilde{\Lambda}_{J,s} = \{\lambda = (j,s,m,p) \in \Lambda_s : |\supp(\psi_{\lambda}) \cap [0,1]^2| \neq 0,  m \in \Z^2, p \leq \frac{J\rho}{2}, j = j_0-1,\dots,J\}.
\eeq
Recall that the support of $f$ is contained in $[0,1]^2$. As the observant reader will have noticed, we bounded the
parameter $p$, which will ensure that we now have ``only'' a finite-dimensional system to handle. Based on this index set,
we next define a system $\Psi_{J,s}$, with the property that $G_s * \Psi_{J,s}$ is a subsystem of $\mathcal{SH}(\varphi_1,\psi_1;g)$.
The reason for this choice will become clear in the next subsection.

\begin{definition}
Letting $\rho \in (0,\wq{\frac{1}{12}})$, $\alpha \ge \frac{6}{\rho}+1$, and $\beta > \alpha + 1$. Further, we
choose $\varphi_1, \psi_1 \in L^2(\R)$ with $(\varphi_1, \psi_1)$ being $(\rho, \alpha, \beta)$-admissible, and
$(G_s)_{s \in \mathbb{S}}$ to form a family of $(\rho, \alpha, \beta; g)$-filters. Let also
$\mathcal{SH}(\varphi_1, \psi_1; g) = \{\psi^{\ell}_{\lambda} : \lambda \in \Lambda_s, s \in \mathbb{S}, \ell = 0,1\}$
denote the associated dualizable shearlet system. For each $J > 0$ and $s \in \mathbb{S}_{J/2}$, we then refer to
\[
\Psi_{J,s} = \{\sigma_{\lambda} := \mathcal{F}^{-1}(\hat \psi_{\lambda}/\hat G_s):\psi_{\lambda} \in \mathcal{SH}(\varphi_1, \psi_1; g), \lambda \in \tilde{\Lambda}_{J,s}\},
\]
where $\tilde{\Lambda}_{J,s}$ is defined as in \eqref{eq:L_sJ}, as the \emph{dualizable shearlet-reconstruction system for limiting scale $J$
and for shear direction $s$.}
\end{definition}

We can immediately derive some basic properties of these systems.

\begin{lemma}\label{lemm:prop_Psi_sJ}
Let $\Psi_{J,s}$ be a shearlet-reconstruction system for limiting scale $J$ and for shear direction $s$.
Then all elements in $\Psi_{J,s}$ are compactly supported, $\Psi_{J,s}$ forms an orthonormal basis for $\spann(\Psi_{J,s})$,
and $\dim(\spann(\Psi_{J,s})) \lesssim 2^{J/2(3+\rho)}$ as $J \to \infty$
\end{lemma}

\begin{proof}
Compact support immediately follows from Proposition \ref{lem:compact}. \gk{Moreover,} Proposition \ref{lemm:ONB}
implies orthogonality, since $\Psi_{J,s} \subset  \Psi_s(\varphi_1,\psi_1)$. The remaining claims follow from the
observation that there exist about $2^{\frac{3}{2}j}2^{p}$ elements $\sigma_{\lambda} \in \Psi_{J,s}$ for each
parameter $j$ and $p \ge 0$, hence the bounds $j \leq J$ and $p \leq \frac{J\rho}{2}$ imply  $\sharp(\Psi_{J,s})
\lesssim 2^{\frac{J}{2}(3+\rho)}.$
\end{proof}

\subsection{Our Sampling-Reconstruction Scheme}\label{subsec:scheme1}

We are now in a position to introduce the reconstruction scheme we choose for $\mathcal{R} = \mathcal{R}(\cE^2(\RR^2),\Delta)$
in \eqref{eq:reconstruction}, with $\Delta = \bigcup_{J > 0} \gk{\{\Delta_J\}}$ as just defined. To develop our scheme, for any
$f \in \cE^2(\RR^2)$ and $J > 0$, we need to define $\mathcal{R}(f,\Delta_J)$. We remark that in fact we will define
$\mathcal{R}(f,\Delta_J)$ based on reconstructions from $\Delta_{J,s}$ for each $s \in \mathbb{S}_{J/2}$.

Now let $f \in \cE^2(\RR^2)$ and fix $J$. Moreover, let $s \in \mathbb{S}_{J/2}$ be arbitrarily fixed for now. We first expand
some $g \in L^2(\RR^2)$ in terms of our chosen reconstruction system $\Psi_{J,s}$. Since this system however only spans its
closed linear hull \gk{and not $L^2(\RR^2)$}, \gk{we only obtain an expansion of the type}
\begin{equation}\label{eq:projection}
P_{\Psi_{J,s}}(g) = \sum_{\lambda \in \tilde{\Lambda}_{J,s}} \ip{g}{\sigma_\lambda}\sigma_\lambda,
\end{equation}
\gk{where $P_{\Psi_{J,s}}$ denotes the orthogonal projection onto $\spann(\Psi_{J,s})$.}
Notice that \gk{the form of} this expansion is only possible due to the fact that $\Psi_{J,s}$ forms an orthogonal basis, see Lemma \ref{lemm:prop_Psi_sJ}.
Next, we set $g := \overline{G}_s * f$ and take the inner product of the previous equation with some element $e_n$, $n \in \ZZ^2$
of the Fourier basis, leading to
\[
\ip{P_{\Psi_{J,s}}(\overline{G}_s * f)}{e_n} = \sum_{\lambda \in \tilde{\Lambda}_{J,s}} \ip{\overline{G}_s * f}{\sigma_\lambda} \ip{\sigma_\lambda}{e_n}.
\]
Recalling that by definition $\sigma_{\lambda} := \mathcal{F}^{-1}(\hat \psi_{\lambda}/\hat G_s)$ and using Plancherel twice,
we obtain
\beq \label{eq:lse}
\ip{P_{\Psi_{J,s}}(\overline{G}_s * f)}{e_n} = \sum_{\lambda \in \tilde{\Lambda}_{J,s}} c_\lambda \ip{\sigma_\lambda}{e_n},
\eeq
where
\beq \label{eq:lse1}
c_\lambda =  \ip{\overline{G}_s * f}{\sigma_\lambda} = \ip{f}{\psi_{\lambda}}.
\eeq

We now claim that in applications, if we have access to $\hat{f}(n)$, we also have access to $\ip{P_{\Psi_{J,s}}(\overline{G}_s * f)}{e_n}$.
For this, first note that in computations we can only compute up to a particular scale. Hence the projection up to scale $J$ for $J$ large
enough is no restriction. Second, the filter $\overline{G}_s$ filters out the part of $f$ which corresponds to $\spann(\Psi_{J,s})$,
hence can also be regarded as only caus\gk{ing} a \gk{small,} controllable error. Finally, by Plancherel, we have
\[
\ip{\overline{G}_s * f}{e_n} = \hat{\overline{G}}_s(n) \cdot \hat{f}(n),
\]
which we indeed can compute knowing the action of the chosen filter $\hat{\overline{G}}_s(n)$ as well as $\hat{f}(n)$.

Applying now our choice for the sampling sets $\Delta_{J,s}$ from Subsection \ref{subsec:samplingsystem}, the linear
system of equations \eqref{eq:lse} becomes
\beq \label{eq:tosolve}
\ip{P_{\Psi_{J,s}}(\overline{G}_s * f)}{e_{n}} = \sum_{\lambda \in \tilde{\Lambda}_{J,s}} c_\lambda \ip{\sigma_\lambda}{e_{n}},
\quad n \in \Delta_{J,s}.
\eeq
Our reconstruction scheme solves this linear system of equations by exploiting the sparse approximation property of our chosen
reconstruction system, i.e., we set
\beq \label{eq:recscheme}
\mathcal{R}(f,\Delta_J) := \sum_{s \in \mathbb{S}_{J/2}} \sum_{\lambda \in \tilde{\Lambda}_{J,s}} \hat{c}_{\lambda} \tilde{\psi}_{\lambda},
\eeq
where
\begin{equation}\label{eq:l1_min}
(\hat{c}_{\lambda})_{\lambda \in \tilde{\Lambda}_{J,s}} = \argmin\|(\tilde{c}_{\lambda})_{\lambda \in \tilde{\Lambda}_{J,s}}\|_1
\mbox{ subject to } \Bigg(\langle P_{\Psi_{J,s}}(\overline{G}_s * f), e_n \rangle = \sum_{\lambda \in \tilde{\Lambda}_{J,s}} \tilde{c}_{\lambda}
\langle \sigma_{\lambda},e_n\rangle \Bigg)_{n \in \Delta_{J,s}}\hspace*{-1cm}.
\end{equation}
In the sequel, instead of \eqref{eq:tosolve}, we will often use the compact notation
\begin{equation}\label{eq:cs_system}
{\bf y_{J,s}} = A_{J,s} {\bf c_{J,s}}
\end{equation}
where ${\bf y_{J,s}} = (\langle P_{\Psi_{J,s}}(\overline{G}_s * f), e_{n} \rangle)_{n \in \Delta_{J,s}}$, $(A_{J,s})_{n,\lambda}
= \langle \sigma_{\lambda}, e_{n}\rangle$, and ${\bf c_{J,s}} = (c_{\lambda})_{\lambda \in \tilde{\Lambda}_{J,s}}$.

We finally wish to point out that in fact the very specific inner structure of the dualizable shearlet systems as well as of the associated
dualizable shearlet-reconstruction systems only enabled this form of the linear system of equations. In fact, each of (DS1)--(DS3) have been
exploited. The sparse approximation property (DS4) will then ensure that $\ell_1$ minimization recovers an approximation with optimally
small asymptotic error.

\subsection{Main Result}

We now have the sampling scheme (\eqref{eq:Delta} and \eqref{eq:samplingsets}) and the associated reconstruction
scheme (\eqref{eq:recscheme} \gk{and \eqref{eq:l1_min})} at hand. The following theorem shows that our scheme is asymptotically optimal
in the sense of Definition \ref{defi:optimal} for almost all cartoon-like functions. In fact, \gk{as announced
in Subsection \ref{subsec:intuition}} we slightly restrict the set of cartoon-like functions to the set
\begin{eqnarray*}
\tilde{\cE}^{2}(\RR^2) & := & \left\{f \in \cE^2(\RR^2) : f \in C^{2,r} \mbox{ smooth with } r \in [1/4,1) \mbox{ everywhere except for points}\right.\\
& & \hspace{0.2cm} \left. \mbox{lying on a } C^2 \mbox{ curve of non-vanishing curvature}\right\}.
\end{eqnarray*}
This condition ensures that the direction of the singularity curve of a cartoon-like function changes with a
certain rate bounded from above, which leads to an almost equal distribution of the significant shearlet
coefficients with respect to the shear $s \in \mathbb{S}_{J/2}$ at a fixed scale $J$. One instance of this can
be seen in Proposition \ref{prop:coeff_decay}(ii).

We now have all ingredients to state the main result of this paper.

\begin{theorem}\label{thm:main}
Letting $\rho \in (0,\frac{2}{13})$, $\alpha \ge \frac{6}{\rho}+1$,
and $\beta > \alpha + 1$, we further choose $\varphi_1, \psi_1 \in L^2(\R)$ with $(\varphi_1, \psi_1)$ being
$(\rho, \alpha, \beta)$-admissible, and $(G_s)_{s \in \mathbb{S}}$ to form a family of $(\rho, \alpha, \beta; g)$-filters.
For each $J > 0$ and $s \in \mathbb{S}_{J/2}$, let also $\Psi_{J,s}$ denote the associated dualizable shearlet-reconstruction
system for limiting scale $J$ and for shear direction $s$. Finally, let the sampling scheme $\Delta$ be chosen as in
\eqref{eq:Delta} and \eqref{eq:samplingsets}, and let the reconstruction scheme $\mathcal{R}$ be defined as in \eqref{eq:recscheme} \gk{and \eqref{eq:l1_min}}.
Then there exists some universal constant $C > 0$ (in particular, independent on $\rho$), such that, for each
$f \in \tilde{\cE}^{2}(\RR^2)$,
\[
\|f - \mathcal{R}(f,\Delta_J)\|_2 \lesssim (\#\Delta_J)^{-1+\rho C} \quad \mbox{as } J \to \infty,
\]
i.e., $\mathcal{R} : \tilde{\cE}^2(\RR^2) \times \Delta \to L^2(\RR^2)$ is arbitrarily close to being asymptotically optimal.
\end{theorem}

Thus, our sampling-reconstruction scheme for reconstruction of the model class of cartoon-like functions from Fourier measurements
is indeed asymptotically optimal, since $\rho$ can be chosen arbitrary small.

\section{Proofs}\label{sec:proof}

This section is devoted to the proof of Theorem \ref{thm:main}. We start with the overall structure and architecture
of this proof, and then slowly delve into the details.

\subsection{Architecture of the Proof of Theorem \ref{thm:main}}

As could already been predicted from the description of our sampling-reconstruction scheme in Subsection \ref{subsec:scheme1}, also
the proof relies heavily on the specific inner structure of dualizable shearlets, more precisely, of the dualizable shearlet-reconstruction
systems for limiting scale $J$ and for shear direction $s$. The proof will first estimate the
error of the minimization procedure in \eqref{eq:l1_min} for each $J$ and $s$, which will be subsequently combined mainly using
the orthogonality stated in Lemma \ref{lemm:prop_Psi_sJ}.

Given $f \in \tilde{\cE}^{2}(\RR^2)$, this splitting is performed by the estimate
\begin{eqnarray}\nonumber
\lefteqn{\|f - \mathcal{R}(f,\Delta_J)\|_2}\\ \label{eq:todo1}
& \lesssim & \sum_{j_0 = 0}^{J}\sum_{s \in \{s(\lceil j^{'}/2 \rceil,q^{'}) \in \mathbb{S}_{J/2} : j^{'} = j_0\}}
\hspace*{-1cm} \|{\bf c_{J,s}}-{\bf \hat{c}_{J,s}}\|^2_2 + \sum_{j \ge J} \sum_{\lambda \in \Lambda_j}|\langle f,\psi_{\lambda}\rangle|^2
+ \sum_{j = 0}^{J} \sum_{\lambda \in \Lambda_j \cap (\Lambda^0_j)^c} |\langle f,\psi_{\lambda}\rangle|^2,
\end{eqnarray}
\wq{where $\Lambda_j$ is the set of all indices $\lambda = (j,s,m,p)\in \Lambda$ with scale $j$ fixed and $\Lambda^0_j$ is the
subset of $\Lambda_j$ where \gk{the} oversampling parameters $p$ associated with $D_p$ are bounded by $\frac{J\rho}{2}$.}

The first term in \eqref{eq:todo1} will be handled by the following proposition.

\begin{proposition}\label{prop:cs_aniso_wavelet}
We retain the same conditions and notations as in Theorem \ref{thm:main}. Let ${\bf c_{J,s}}  := (c_{\lambda})_{\lambda \in \tilde{\Lambda}_{J,s}}
= (\ip{f}{\psi_{\lambda}})_{\lambda \in \tilde{\Lambda}_{J,s}}$ as in \eqref{eq:lse1} and ${\bf \hat{c}_{J,s}} := (\hat{c}_{\lambda})_{\lambda \in
\tilde{\Lambda}_{J,s}}$ be the solution of the $\ell_1$ minimization problem \eqref{eq:l1_min}. Then with probability at least $1-2^{-J}$, we have
\[
\|{\bf c_{J,s}}-{\bf \hat{c}_{J,s}}\|_{2} \lesssim \frac{\sigma_{n_{J,s}}({\bf c}_{J,s})_1}{\sqrt{n_{J,s}}} \quad \text{as} \,\, J \rightarrow \infty
\]
where
\[
\sigma_{n_{J,s}}({\bf c_s})_1 := \inf \{\|{\bf c}_s-{\bf \tilde{c}_s}\|_1 : \|{\bf \tilde{c}_s}\|_0 \leq n_{J,s}\}
\]
and $n_{J,s} \sim J\cdot2^{\frac{J-j_0}{2}}2^{\wq{2J\rho}}$ for each $s = s(\lceil j_0/2 \rceil,q_0) \in \mathbb{S}_{J/2}$.
\end{proposition}

One ingredient of the proof of this proposition, provided in Subsection \ref{subsec:prop2}, is a more or less classical result from compressed
sensing. We state here the version presented in \cite{KW13}.

\begin{theorem}[\cite{KW13}]\label{thm:candes}
Let $A \in \mathbb{C}^{m \times N}$ and $x \in \mathbb{C}^N$. Suppose that $y = Ax$ where the restricted isometry constant of $A$ satisfies
$\delta_{5s} < \frac{1}{3}$. Then the solution to
\[
\min_{z \in \mathbb{C}^N} \|z\|_1 \quad \text{subject to} \quad Az = y
\]
obeys
\[
\|\hat x - x\|_2 \lesssim \frac{\sigma_s(x)_1}{\sqrt{s}}.
\]
\end{theorem}

To apply this result, it needs to be shown that for the matrix $\tilde{A}_{J,s}$ whose entries are given by
\[
\frac{\langle \sigma_{\lambda}, e_{n} \rangle}{\sqrt{m_{J,s} \cdot p_{J,s}(n)}}, \quad n \in \Delta_{J,s}, \lambda \in \tilde{\Lambda}_{J,s},
\]
(compare the discussions in Subsections \ref{subsec:generalapproach} and \ref{subsec:frames} as well as the definition of the
$\ell_1$ minimization problem \eqref{eq:l1_min}), its associated restricted isometry constant
\[
\hat{\delta}_{n_{J,s}} := \max_{\sharp(S) \leq n_{J,s}} \|\tilde{A}^*_S\tilde{A}_S - \mathrm{Id}\|_2
\]
can be  made arbitrarily small. For this, a few ideas from \wq{\cite{KW13} and \cite{Rau10}} can be adapted and used.
However, one main difficulty in our setting is \wq{to determine \gk{the} required sampling size $m_{J,s}$ and \gk{the} sparsity level
$n_{J,s}$ for each shearing parameter $s$ by exploring sparsity patterns in the shearlet representation.}

The second and third term in \eqref{eq:todo1} require careful control of the decay of the shearlet coefficients. One key condition
is that \gk{standard} shearlets generated by $A_j$ and $S_k$ are supported in $S^{-1}_sA^{-1}_j[-C,C]^2$ for some $C > 0$
with $s = \frac{k}{2^{\lceil j/2 \rceil}}$. In this area, the $C^2$ curvilinear singularity of a cartoon-like function is well
approximated by its tangent. However, one can easily check that $\mbox{supp}(\psi_{\lambda}) \subset S^{-1}_sA^{-1}_{j_0}[-C,C]^2$.
In fact, $\mbox{supp}(\psi_{\lambda})$ is essentially the same as the support of \wq{standard} shearlets \wq{considered in \cite{KL11}} for scale $j = j_0$. However,
$\mbox{supp}(\psi_{\lambda})$ is much larger when $j \gg j_0$. To resolve this issue, as already done in proofs in \cite{KL15},
we will approximate $\psi_{\lambda}$ by more suitable functions of smaller supports comparable to the size of the supports of
\wq{standard} shearlets with controllable error bound as follows:

For $\psi_{\lambda} = G_s * \psi_{j,s,m,p}$ with $\lambda = (j,s,m,p) \in \Lambda, s = s(\lceil j_0/2 \rceil,q_0)  \in \mathbb{S}$ and
$j \ge j_0$, define
\begin{equation}\label{eq:split_shearlet}
\hat{\psi}^{\sharp}_{\lambda}(\xi) = \sum_{j^{'} = \max(j(1-\rho),j_0)}^{\infty}|\hat g (A^{-*}_{j^{'}}S^{-*}_s \xi)|^2 \hat\psi_{j,s,m,p}(\xi).
\end{equation}
Moreover, for $\psi_{\lambda} \in \mathcal{SH}(\varphi_1,\psi_1;g)$ of the form  $\psi_{\lambda} = G_s * \varphi_{j,s,m,p}$, we define
$\psi^{\sharp}_{\lambda}$ as in \eqref{eq:split_shearlet} except for replacing $\hat \psi_{j,s,m,p}$ by $\hat \varphi_{j,s,m,p}$.
Observe that for sufficiently small $\rho > 0$, $\mbox{supp}({\psi}^{\sharp}_{\lambda})$ approximates the support of the original
shearlet considered in \cite{KL11} as follows.

\begin{proposition}[\cite{KL15}]\label{prop:extra_shearlets}
Let $\psi^{\sharp}_{\lambda} \in L^2(\R^2)$ as in \eqref{eq:split_shearlet}. Then we have
\[
\supp(\psi^{\sharp}_{\lambda}) \subset S^{-1}_s A^{-1}_j \Bigl([-2^{j\rho}C,2^{j\rho}C] \times [-2^{j/2\cdot\rho}C, 2^{j/2\cdot\rho}C]+D_pm\Bigr) \quad \text{for some} \quad C>0.
\]
Furthermore, we have
\[
|\langle f,(\psi_{\lambda}-\psi^{\sharp}_{\lambda})\rangle| \lesssim \wq{2^{-j\rho \alpha}2^{-p\alpha}} \|f\|_2 \quad \text{for} \quad f \in L^2(\R^2).
\]
\end{proposition}

The decay estimates for the shearlet coefficients $|\langle f,\psi_{\lambda}\rangle|$ stated in the next result
with (i) already being proven in \cite{KL15} and (ii) in Subsection \ref{subsec:prop1}, will be the essential
ingredients for controlling \wq{the first and second terms} in \eqref{eq:todo1}.

\begin{proposition}[\cite{KL15}]\label{prop:coeff_decay}
Assume that $f$ is a cartoon-like function with $C^2$ discontinuity curve given by $x_1 = E(x_2)$. For $\psi_{\lambda} \in \mathcal{SH}(\varphi_1,\psi_1;g)$ with $\lambda = (j,s,m,p) \in \Lambda$ and $j \ge 0$, define $\psi^{\sharp}_{\lambda} \in L^2(\R^2)$ as in \eqref{eq:split_shearlet}. Let $\hat{x}_2 \in \R$ so that $(E(\hat{x}_2),\hat{x}_2) \in \mathrm{supp}(\psi^{\sharp}_{\lambda})$ and $\hat s = E^{'}(\hat{x}_2)$. Also let $k_s \in \Z$ so that $s = k_s/2^{j/2}$.
Then for $\lambda = (j,s,m,p) \in \Lambda$ with $j \ge 0$, \gk{the following hold.}
\begin{itemize}
\item[(i)] If $|\hat s| \leq 3$, then
\[|\langle f,\psi_{\lambda}\rangle| \lesssim \min\Bigl(2^{-\frac{3}{4}j}, \frac{2^{-\frac{3}{4}j}2^{3\rho j}}{|k_s + 2^{j/2}\hat{s}|^3}\Bigl).\]
\item[(ii)] If $h \in L^2(\R^2)$ is $C^{2,r}$ smooth in $\supp(\psi^{\sharp}_{\lambda})$ with $\|h\|_{C^{2,r}(\supp(\psi^{\sharp}_{\lambda}))}
\leq C$ and $r \in [1/4,1)$ for some $C>0$, then
\[
|\langle h,\psi_{\lambda}\rangle| \lesssim 2^{p/2}2^{-(3-4\rho)j}.
\]
\end{itemize}
\end{proposition}

\wq{Comparing with the \gk{standard} shearlet\gk{s} considered in \cite{KL11}, our dualizable shearlet system $\mathcal{SH}(\varphi_1,\psi_1;g)$
has \gk{an} additional parameter $p$ associated with \gk{the} oversampling matrix $D_p$. In addition to decay estimates in Proposition \ref{prop:coeff_decay},
we \gk{therefore} need the decay \gk{of} the shearlet coefficients with respect to $p$. The following \gk{p}roposition shows how shearlet
coefficients $\langle f,\psi_{\lambda}\rangle$ with this additional parameter $p$ can be controlled, which will be \gk{a} crucial ingredient
for controlling the third term in \eqref{eq:todo1}.

\begin{proposition}[\cite{KL15}]\label{prop:control_p}
For $f \in L^2(\R^2)$, we have
\[
|\langle f,\psi_{\lambda}\rangle| \lesssim 2^{-\frac{\alpha}{2} p} \cdot \|f\|_2 \quad \text{for all } \lambda = (j,s,p,m) \in \Lambda.
\]
\end{proposition}
}
Let us finally mention that in the subsequent proofs, without loss of generality, we will only consider shearlet elements
$\psi^{0}_{\lambda} \in \mathcal{SH}(\varphi_1, \psi_1; g)$ associated with one frequency cone. Since the elements
$\psi^{1}_{\lambda}$ \wq{are given by switching the variables with $R$}, they can be dealt with similarly. Hence, for the sake of brevity,
we will often omit the superscript ``0'', i.e., we write
\[
\psi_{\lambda} := \psi^{0}_{\lambda} \quad \mbox{and} \quad \tilde{\psi}_{\lambda} := \tilde{\psi}^{0}_{\lambda}.
\]
Moreover, for technical reasons, again without loss of generality, we now rescale each element $\sigma_{\lambda} \in \Psi_{J,s}$
so that
\[
\supp(\sigma_{\lambda}) \subset [0,1]^2, \quad \mbox{for all } \lambda \in \tilde \Lambda_{J,s}.
\]

\subsection{Useful Lemmata from \cite{Rau10}}

For the convenience of the reader, we state two lemmata both from \cite{Rau10}, which will be required for the proof
of Proposition \ref{prop:cs_aniso_wavelet}.

\begin{lemma}[\cite{Rau10}]\label{lem:lem02}
Assume that $X = (X_\ell)_{\ell = 1}^{m}$ is a sequence of independent random vectors in $\C^N$ with
$\|X_{\ell}\|_{\infty} \leq K$ for $\ell = 1,\dots,m$. Then, for $1 \leq p < \infty$ and $n \leq m$,
\[
\Bigl( \mathbb{E}\vertiii{\sum_{\ell = 1}^{m}(X^*_iX_{\ell} - \mathbb{E}X^*_{\ell}X_{\ell})}^p_{n}\Bigr)^{\frac{1}{p}}
\leq 2 \Bigl(\mathbb{E}\vertiii{\sum_{\ell = 1}^{m}\epsilon_{\ell} X^*_{\ell}X_{\ell}}_n^p \Bigr)^{\frac{1}{p}},
\]
where $\epsilon = (\epsilon_{\ell})_{\ell = 1}^{m}$ is a Rademacher sequence independent of $X$.
Furthermore, for $p \ge 2$,
\[
\Bigl( \mathbb{E}\vertiii{\sum_{\ell = 1}^{m}\epsilon_{\ell} X_{\ell}X^*_{\ell}}_n^p\Bigr)^{\frac{1}{p}}
\leq C{\eta}^{\frac{1}{p}}K\sqrt{p}\sqrt{n}\log(100n)\sqrt{\log(4N)\log(10m)}\sqrt{\vertiii {\sum_{\ell=1}^{m}X_{\ell}X^*_{\ell}}_n},
\]
where $C$ and $\eta$ are some positive constants and  $\eta < 7$.
\end{lemma}

\begin{lemma}\label{lem:lem01} \cite{Rau10}\,\,
Suppose that $Z$ is a random variable satisfying
\[
(\mathbb{E}|Z|^p)^{\frac{1}{p}} \leq C_1C_2^{\frac{1}{p}}p^{\frac{1}{r}} \quad \text{for all}\,\, p \ge p_0
\]
for some $C_1$, $C_2$, $r$, $p_0 > 0$. Then
\[
\mathbb{P}(|Z| \ge e^{\frac{1}{r}}C_1 u) \leq C_2 e^{-u^r/r} \quad \text{for all}\,\, u \ge p_0^{\frac{1}{r}}.
\]
\end{lemma}

\subsection{Proof of Preparatory Results}

\subsubsection{Proof of Proposition \ref{prop:coeff_decay}(ii)} \label{subsec:prop1}

First notice that, by the error estimate in Proposition \ref{prop:extra_shearlets}, it suffices to show that
\beq \label{eq:todo}
|\langle h,\psi^{\sharp}_{\lambda}\rangle| \lesssim 2^{p/2}2^{-(3-4\rho)j} \quad \mbox{for } \lambda = (j,0,0,p) \in \Lambda.
\eeq
Hence, let $\lambda = (j,0,0,p)$. Again by Proposition \ref{prop:extra_shearlets}, there exists some $L > 0$ with
\begin{equation}\label{eq:support_cond}
\supp(\psi^{\sharp}_{\lambda}) \subset [-2^{-j+\rho j}L,2^{-j+\rho j}L] \times [-2^{-j/2+\rho j/2}L,2^{-j/2+\rho j/2}L].
\end{equation}
Further, by construction, $\psi^{\sharp}_{\lambda}$ satisfies the vanishing moment condition
\[
\int_{\R} x_1^{\ell} \cdot \psi^{\sharp}_{\lambda}(x_1,x_2) dx_1 = 0 \quad \text{for} \,\, \ell = 0,1,2 \quad
\text{and} \quad \|\psi^{\sharp}_{\lambda}\|_{\infty} \lesssim 2^{\frac{3}{4}j}2^{p/2}.
\]
Now let $h \in C^{2,r}(\R^2)$, and observe that we may write $h$ as
\begin{equation}\label{eq:three}
h(x_1,x_2) = \sum_{\ell = 0}^{2} \Bigl(\frac{1}{\ell !}\Bigr)\Bigl(\frac{\partial}{\partial x_1}\Bigr)^{\ell}
h(0,x_2){x_1}^{\ell} + R(x_1,x_2)x_1^{9/4} \quad \text{for each} \,\, (x_1,x_2) \in \supp(\psi^{\sharp}_{\lambda})
\end{equation}
for some bounded function $\gk{g}$ with $\|\gk{g}\|_{\infty} \leq C$ for some $C>0$. Finally, by \eqref{eq:support_cond}--\eqref{eq:three},
we obtain
\begin{eqnarray*}
|\langle h,\psi^{\sharp}_{\lambda}\rangle| &\lesssim& \|\gk{g}\|_{\infty}\|\psi^{\sharp}_{\lambda}\|_{\infty}
\int_{-2^{-j/2(1-\rho)}L}^{2^{-j/2(1-\rho)}L}\int_{-2^{-j(1-\rho)}L}^{2^{-j(1-\rho)}L} |x_1|^{9/4} dx_1 dx_2 \\
&\lesssim& 2^{\frac{3}{4}j}2^{p/2}2^{-j/2(1-\rho)}2^{-j\frac{13}{4}(1-\rho)} \leq 2^{\frac{p}{2}}2^{-(3-4\rho)j}
\end{eqnarray*}
This proves \eqref{eq:todo}, and hence the claim. \hfill \qed

\subsubsection{Proof of Proposition \ref{prop:cs_aniso_wavelet}} \label{subsec:prop2}

\gk{The main part} of this proof will be to show that Proposition \ref{thm:candes} can be applied \gk{in a particular way}. For this, let $s = s(\lceil j_0/2 \rceil,q_0)
\in \mathbb{S}_{J/2}$, and set $\Delta_{J,s} =: \{t_1, \ldots, t_{m_{J,s}}\}$. Moreover, let $A$ be the matrix defined by
\[
(\tilde{A}_{J,s})_{i,\lambda} := \frac{\hat{\sigma}_{\lambda}(t_i)}{\sqrt{m_{J,s} \cdot p_{J,s}(t_i)}} \quad \text{for}\,\, \lambda
\in \tilde{\Lambda}_{J,s} \,\, \text{and} \,\, i = 1, \ldots, m_{J,s},
\]
where $\hat{\sigma}_{\lambda}(t_i) = \langle \sigma_{\lambda}, e_{t_i} \rangle$ and the $t_i$'s are drawn from the probability
density function $p_{J,s}$. One should compare this choice with \eqref{eq:cs_system}, seeing that it deviates from $A$ by $(m_{J,s} \cdot p_{J,s}(t_i))^{1/2}$.
Further, let $S$ be any set of the column indices of $\tilde{A}$ with $\sharp(S) \leq n_{J,s}$, and $\tilde{A}_S$ be the column submatrix of $\tilde{A}$
consisting of columns indexed by $S$. Observe that $(B^T_S)^T_S$ is the submatrix of $B$ consisting of columns and rows indexed by
$S$, when $B$ is a square matrix.

We will now prove that the restricted isometry constant of $\tilde{A}$ given by
\[
\hat{\delta}_{n_{J,s}} := \max_{\sharp(S) \leq n_{J,s}} \|\tilde{A}^*_S\tilde{A}_S - \mathrm{Id}\|_2
\]
can be made arbitrarily small. Applying Proposition \ref{thm:candes} then proves the claim.

For this, for each $i = 1, \ldots, m_{J,s}$, first set
\[
X_{i} := \Big(\frac{\hat{\sigma}_{\lambda}(t_i)}{\sqrt{p_{J,s}(t_i)}}\Big)_{\lambda \in \tilde{\Lambda}_{J,s}} \in \wq{\C^{\tilde{M}_{s}}},
\]
where $\wq{\tilde{M}_{s}} = \sharp(\tilde{\Lambda}_{J,s}) \lesssim 2^{\frac{J}{2}(3+\rho)}$. We further require the matrix norm
\[
\vertiii{B}_n := \sup_{z \in D^2_{n,d}} |\langle Bz,z\rangle|, \quad \mbox{where }  D^2_{n,d} := \{z \in \C^d : \|z\|_2 \leq 1, \|z\|_0 \leq n\},
\, n \ge 0.
\]
With this, we can estimate $\hat{\delta}_{n_{J,s}}$ by
\begin{eqnarray}\nonumber
\hat{\delta}_{n_{J,s}} &=& \vertiii{\frac{1}{m_{J,s}}\sum_{i=1}^{m_{J,s}}X_{i}X^*_{i}-\mathrm{Id}}_{n_{J,s}} \\ \label{eq:est1}
&\leq& \frac{1}{m_{J,s}}\vertiii{\sum_{i = 1}^{m_{J,s}} (X_{i}X^*_{i}-\mathbb{E}X_{i}X^*_{i})}_{n_{J,s}} + \vertiii{\mathbb{E}X_{i}X^*_{i}-\mathrm{Id}}_{n_{J,s}}
=: T_1 + T_2.
\end{eqnarray}

We now show that both $T_1$ and $T_2$ can be made arbitrarily small. We start with $T_2$, and bound each entry in the matrix
$\mathrm{Id} - \mathbb{E}X_{i}X^*_{i}$ by
\beq \label{eq:est2}
\Big|\Bigl( \mathrm{Id} - \mathbb{E}X_{i}X^*_{i}\Bigr)_{\lambda,\lambda^{'} \in \tilde{\Lambda}_{J,s}}\Big|
= \Big|\sum_{n \in \Omega^c_J} \hat{\sigma}_{\lambda}(n)\overline{\hat{\sigma}_{\lambda^{'}}(n)} \Big|
\leq \Bigl( \sum_{n \in \Omega^c_J} |\hat{\sigma}_{\lambda}(n)|^2\Bigr)^{1/2}\Bigl( \sum_{n \in \Omega^c_J} |\hat{\sigma}_{\lambda^{'}}(n)|^2\Bigr)^{1/2},
\eeq
where $\Omega^c_J = \Omega_1 \cup \Omega_2 \cup \Omega_3$ with
\begin{eqnarray*}
\Omega_1 &=& \{n \in \Omega^c_J : |n_1| > 2^{J(1+\rho)}\}, \\
\Omega_2 &=& \{n \in \Omega^c_J : |n_2| > 2^{J(1+\rho)} \,\, \mathrm{and} \,\, |n_1| \leq 2^{J(1+\rho/2)}\}, \mbox{ and} \\
\Omega_3 &=& \{n \in \Omega^c_J : |n_2| > 2^{J(1+\rho)} \,\, \mathrm{and} \,\, |n_1| > 2^{J(1+\rho/2)}\}.
\end{eqnarray*}
Next, for each  $\lambda = (j,s,m,p) \in \tilde{\Lambda}_{J,s}$, the decay conditions in Definition \ref{defi:g} imply
\beq \label{eq:est3}
\sum_{n \in \Omega^c_J} |\hat{\sigma}_{\lambda}(n)|^2 \lesssim \sum_{n \notin \Omega_J}\frac{2^{-\frac{3}{2}j-p}}{(1+|2^{-j}n_1|)^{2\beta}
(1+|2^{-j/2-p}(n_2 - s n_1)|)^{2\beta}}.
\eeq
Using the decomposition of $\Omega^c_J$, we first observe that, since $|n_2-sn_1| \ge ||n_2|-|n_1|| \ge 2^{J(1+\rho/2)}$
for all $n \in \Omega_2$, we have
\begin{eqnarray*}
\sum_{n \in \Omega_2} \frac{2^{-\frac{3}{2}j-p}}{(1+|2^{-j}n_1|)^{2\beta}
(1+|2^{-j/2-p}(n_2 - s n_1)|)^{2\beta}} &\lesssim& \sum_{|m| \ge 2^{J(1+\rho/2)}}
\frac{2^{-j/2-p}}{(1+|2^{-j/2-p}m|)^{2\beta}} \\
&\lesssim& 2^{\frac{J}{2}(1-2\beta)}.
\end{eqnarray*}
More\gk{over}, we have
\[
\sum_{n \in \Omega_1 \cup \Omega_3} \frac{2^{-\frac{3}{2}j-p}}{(1+|2^{-j}n_1|)^{2\beta}
(1+|2^{-j/2-p}(n_2 - s n_1)|)^{2\beta}} \lesssim \sum_{|n_1| > 2^{J(1+\rho/2)}}
\frac{2^{-j}}{(1+|2^{-j}n_1|)^{2\beta}} \lesssim 2^{J\rho(1-2\beta)}.
\]
Applying the last two estimates to \eqref{eq:est3} and inserting into \eqref{eq:est2}, we obtain
\[
\Big|\Bigl( \mathrm{Id} - \mathbb{E}X_{\ell}X^*_{\ell}\Bigr)_{\lambda,\lambda^{'}}\Big| \lesssim 2^{\frac{J\rho}{2}(1-2\beta)}.
\]
If $\beta \ge \frac{1}{2}(\frac{2}{\rho}+1)$, then
\[
\sup_{\lambda,\lambda^{'} \in \tilde{\Lambda}_{J,s}}
\Big|\Bigl( \mathrm{Id} - \mathbb{E}X_{\ell}X^*_{\ell}\Bigr)_{\lambda,\lambda^{'}}\Big|
\lesssim 2^{-J},
\]
which in turn implies that
\begin{eqnarray}\label{eq:tail}
T_2 = \vertiii{\mathbb{E}X_{i}X^*_{i}-\mathrm{Id}}_{n_{J,s}}
&\leq& \max_{\sharp(S) \leq n_{J,s}}\|((\mathbb{E}X_{i}X^*_{i}-\mathrm{Id})^T_S)^T_S\|_2 \nonumber \\
&\leq& \max_{\sharp(S) \leq n_{J,s}}\sqrt{\|((\mathbb{E}X_{i}X^*_{i}-\mathrm{Id})^T_S)^T_S\|_{\infty}\|((\mathbb{E}X_{i}X^*_{i}-\mathrm{Id})^T_S)^T_S\|_1} \nonumber \\
&\lesssim& \sqrt{2^{-J}2^{J/2(1+\wq{5\rho})}} = 2^{-J/4(1-\wq{5\rho})}.
\end{eqnarray}
For the last inequality, we used $n_{J,s} \lesssim 2^{J/2(1+\wq{5\rho})}$.

Next, we estimate $T_1$ in \eqref{eq:est1}. For this, let $n = (n_1,n_2) \in \Omega_J$, and first consider
\begin{eqnarray}\label{eq:ubound}
\sup_{\lambda \in \tilde{\Lambda}_{J,s}} \frac{|\hat{\sigma}_{\lambda}(n)|^2}{|p_{J,s}(n)|}
&\lesssim&
\Biggl(\sup_{j \ge 0} \frac{2^{-j}}{(1+|2^{-j}n_1|)^{2\beta}}\sup_{j,p \ge 0}\frac{2^{-j/2-p}}{(1+|2^{-j/2-p}(n_2-sn_1)|)^{2\beta}}\Biggr)\Biggl(\frac{1}{|p_{J,s}(n)|}\Biggr) \nonumber \\
&\lesssim&
\Biggl(\frac{1}{(1+|n_1|)(1+|n_2-sn_1|)}\Biggl)\Biggl(J^2(1+|n_1|)(1+|n_2-sn_1|)\Biggr)
\nonumber \\ &\lesssim&
J^2.
\end{eqnarray}
Thus, by definition of $X_i$ we conclude that there exists some $K > 0$ with
\[
\|X_i\|_{\infty} \leq K\cdot J \quad \mbox{for } i = 1,\dots,m_{J,s}.
\]
For the sake of brevity, for $p \ge 2$, we set
\[
E_p := (\mathbb{E}T_1^p)^{\frac{1}{p}} \quad \mbox{and} \quad
D_{p,n_{J,s},m_{J,s}} := C\cdot K \cdot{\eta}^{\frac{1}{p}}J\sqrt{p}\sqrt{n_{J,s}}\log(100n_{J,s})\sqrt{\log(4\tilde{M}_s)\log(10m_{J,s})},
\]
where $\eta$ is some constant \gk{in} the second claim of Lemma \ref{lem:lem02}, and $J$ is chosen sufficiently large so that
$T_2 \leq 1$, which is possible by \eqref{eq:tail}. Then, by Lemma \ref{lem:lem02}, \eqref{eq:tail}, and \eqref{eq:ubound},
{\allowdisplaybreaks
\begin{eqnarray*}
E^p_p &\leq& \Bigl(\frac{2D_{p,n_{J,s},m_{J,s}}}{\sqrt{m_{J,s}}}\Bigr)^p \mathbb{E}\vertiii{\frac{1}{m_{J,s}}\sum_{\ell = 1}^{m_{J,s}}X_{\ell}X^*_{\ell}}^{\frac{p}{2}}_{n_{J,s}} \\
&\leq& \Bigl(\frac{2D_{p,n_{J,s},m_{J,s}}}{\sqrt{m_{J,s}}}\Bigr)^p \mathbb{E}\Bigl( \vertiii{\frac{1}{m_{J,s}}\sum_{\ell = 1}^{m_{J,s}}X_{\ell}X^*_{\ell}-\mathrm{Id}}_{n_{J,s}}+1\Bigr)^{\frac{p}{2}} \\
&\leq& \Bigl(\frac{2D_{p,n_{J,s},m_{J,s}}}{\sqrt{m_{J,s}}}\Bigr)^p \mathbb{E} \Bigl( \vertiii{\frac{1}{m_{J,s}}\sum_{\ell = 1}^{m_{J,s}}(X_{\ell}X^*_{\ell} - \mathbb{E}X_{\ell}X^*_{\ell})}_{n_{J,s}} + T_2 +1 \Bigr)^{\frac{p}{2}} \\
&\leq& \Bigl(\frac{2D_{p,n_{J,s},m_{J,s}}}{\sqrt{m_{J,s}}}\Bigr)^p \mathbb{E} \Bigl( \vertiii{\frac{1}{m_{J,s}}\sum_{\ell = 1}^{m_{J,s}}(X_{\ell}X^*_{\ell} - \mathbb{E}X_{\ell}X^*_{\ell})}_{n_{J,s}} + 2  \Bigr)^{\frac{p}{2}}.
\end{eqnarray*}
}
From this, we have
\[
E_p \leq \frac{2D_{p,n_{J,s},m_{J,s}}}{\sqrt{m_{J,s}}}\sqrt{\mathbb{E}\frac{1}{m_{J,s}}\vertiii{\sum_{\ell = 1}^{m_{J,s}}(X_{\ell}X^*_{\ell} - \mathbb{E}X_{\ell}X^*_{\ell})}_{n_{J,s}} + 2}
\leq \frac{2D_{p,n_{J,s},m_{J,s}}}{\sqrt{m_{J,s}}} \sqrt{E_p + 2}.
\]
This yields
\[
E^2_p \leq \Biggl(\frac{2D_{p,n_{J,s},m_{J,s}}}{\sqrt{m_{J,s}}}\Biggr)^2E_p+ \Biggl(\frac{2D_{p,n_{J,s},m_{J,s}}}{\sqrt{m_{J,s}}}\Biggr)^2 2
\]
and also
\[
\Biggl( E_p - \frac{1}{2}\Biggl(\frac{2D_{p,n_{J,s},m_{J,s}}}{\sqrt{m_{J,s}}}\Biggr)^2\Biggr)^2  \leq \Biggl(\frac{2D_{p,n_{J,s},m_{J,s}}}{\sqrt{m_{J,s}}}\Biggr)^2 \Biggl( 2 + \frac{1}{4}\Biggl(\frac{2D_{p,n_{J,s},m_{J,s}}}{\sqrt{m_{J,s}}}\Biggr)^2\Biggr),
\]
implying that
\[
E_p \leq \frac{2D_{p,n_{J,s},m_{J,s}}}{\sqrt{m_{J,s}}}\Biggl(2 +\frac{1}{4}\Biggl(\frac{2D_{p,n_{J,s},m_{J,s}}}{\sqrt{m_{J,s}}}\Biggr)^2
\Biggr)^{1/2} + \frac{1}{2}\Biggl(\frac{2D_{p,n_{J,s},m_{J,s}}}{\sqrt{m_{J,s}}}\Biggr)^2.
\]
Assuming $\frac{2D_{p,n_{J,s},m_{J,s}}}{\sqrt{m_{J,s}}} \leq \frac{1}{2}$, we conclude that
\[
E_p \leq \kappa \frac{2D_{p,n_{J,s},m_{J,s}}}{\sqrt{m_{J,s}}} \quad \mbox{with }\kappa = \frac{\sqrt{33}+1}{4}.
\]
This implies
\[
(\mathbb{E} \min (1/2,T_1^p)\gk{)}^{\frac{1}{p}} \leq   (\min((1/2)^p,\mathbb{E}T_1^p))^{\frac{1}{p}}
\leq \min((1/2),E_p)
\leq \kappa\frac{2D_{p,n_{J,s},m_{J,s}}}{\sqrt{m_{J,s}}}.
\]
Now, by Lemma \ref{lem:lem01} (with $r =2, p_0 = 2$) and using that $\eta < 7$, for all $u \ge 2$, we obtain
\[
\mathbb{P}\Bigl(\min(1/2,T_1) \ge e^{\frac{1}{2}} (2\kappa) C \cdot KJ \sqrt{\frac{n_{J,s}}{m_{J,s}}}\log(100n_{J,s})\sqrt{\log(4\tilde{M}_s)\log(10m_{J,s})}u\Bigr) < 7e^{-\frac{u^2}{2}}.
\]
This implies that if
\begin{equation}\label{eq:num_measure}
\frac{m_{J,s}}{\log(m_{J,s})} \ge 2 n_{J,s} (e^{\frac{1}{2}}2\kappa C)^2\delta^{-2}K^2J^2\log^2(100n_{J,s})\log(7\epsilon^{-1})\log(4\tilde{M}_s) =: T_3,
\end{equation}
then, with probability at least $1-\epsilon$,
\beq \label{eq:est5}
T_1 \leq \delta \leq \frac{1}{2}
\eeq

Let now $\epsilon = 2^{-J}$ and note that $\tilde{M}_s \lesssim 2^{\frac{J}{2}(3+\rho)}$. Since $n_{J,s} \sim J\cdot 2^{\frac{J-j_0}{2}}2^{\wq{2J\rho}}$
and $m_{J,s} \sim 2^{\frac{J-j_0}{2}2^{\wq{3}J\rho}}$, we have
\begin{equation}\label{eq:num_measure02}
\frac{1}{T_3} \cdot \frac{m_{J,s}}{\log(m_{J,s})} \ge \tilde{C}\Bigl(\frac{\delta^2}{J^8}\Bigr)2^{\wq{J\rho}} =: T_4
\end{equation}
with some constant $\tilde{C} > 0$. Therefore, for any $\delta > 0$ and $\tilde{C} > 0$ in \eqref{eq:num_measure02}, one may choose
a sufficiently large $J>0$ such that $T_4 \ge 1$. This proves \eqref{eq:num_measure}, and hence \eqref{eq:est5}.

Estimating \eqref{eq:est1} by \eqref{eq:tail} and \eqref{eq:est5}, implies that the restricted isometry constant of $\tilde{A}$ can
indeed be made arbitrary small by choosing \gk{$J$} sufficiently large. As discussed before, by Theorem \ref{thm:candes}, the proposition
is proved. \hfill \qed

\subsection{Proof of Theorem \ref{thm:main}}

We start by defining dyadic cubes $Q_{j,\ell}$ as
\[
Q_{j,\ell} := 2^{-j/2}[0,1]^2+2^{-j/2}\ell \quad \text{for }  \ell \in \Z^2.
\]
Those cubes intersecting the discontinuity curve $\Gamma$ are of particular interest, hence we set
\[
\mathcal{Q}_j := \{Q_{j,\ell} : \mathrm{int}(Q_{j,\ell}) \cap \Gamma \neq \emptyset\},
\]
where $\mathrm{int}(Q_{j,\ell})$ is the interior set of $Q_{j,\ell}$. We assume that the discontinuity curve
$\Gamma$ is given by $x_1 = E(x_2)$ with $E \in C^2([0,1])$ and only consider this case. In fact, for sufficiently
large $j$, the discontinuity curve $\Gamma$ can be expressed as either $x_1 = E(x_2)$ or $x_2 = \tilde E(x_1)$ in
$Q_{j,\ell} \in \mathcal{Q}_j$ and the same arguments can be applied for $x_2 = \tilde E(x_1)$ except for switching
the order of variables.

For each $Q_{j,\ell} \in \mathcal{Q}_j$, let $E_{j,\ell}$ be a $C^2$ function such that
\[
\Gamma \cap \mathrm{int}(Q_{j,\ell}) = \{(x_1,x_2) \in \mathrm{int}(Q_{j,\ell}) : x_1 = E_{j,\ell}(x_2)\}.
\]
Again, for sufficiently large $j$ and each $Q_{j,\ell} \in \mathcal{Q}_j$, we may assume either
\[
\gk{\|E^{'}_{j,\ell}\|_{\infty} \leq 3 \quad \mbox{or} \quad} \inf_{(x_1,x_2) \in \mathrm{int}(Q_{j,\ell})}|E^{'}_{j,\ell}(x_2)| > 3/2.
\]
For this, we \wq{only consider the case when}  $\|E^{'}_{j,\ell}\|_{\infty} \leq 3$ for any $Q_{j,\ell} \in \mathcal{Q}_j$ and \gk{a}
similar argument can be applied for the latter case. Moreover, we define the orientation of the discontinuity curve $\Gamma$ in each dyadic cube $Q_{j,\ell}$ by
\begin{equation}\label{eq:local_slope}
\hat{s}_{j,\ell} = E^{'}_{j,\ell}(\hat{x}_2) \quad \text{for some} \,\, (E_{j,\ell}(\hat{x}_1),\hat{x}_2) \in \mathrm{int}(Q_{j,\ell}) \cap \Gamma.
\end{equation}

We now aim to show that Proposition \ref{prop:cs_aniso_wavelet} can be applied, which will provide one ingredient to the estimate for
the error $\|f-\hat{f}\|_2$. For this, let $J > 0$ be fixed. For
each $s = s(\lceil j_0/2 \rceil,q_0)\in S_{J/2}$ and $j \ge j_0$, we choose $k_{j,s} \in \Z$ so that $s = \frac{k_{j,s}}{2^{j/2}}$.
Next, for each $\ell_2 \in \{0,\dots,2^{j/2}-1\}$ with $j \ge j_0$, there exists some point $\xi_{\ell_2} \in [0,1]$ such that
\[
E^{'}(\ell_2/2^{j/2}) = E^{'}(0) + E^{''}(\xi_{\ell_2})\frac{\ell_2}{2^{j/2}}.
\]
Notice that we may assume $\hat{s}_{j,\ell} = E^{'}(\ell_2/2^{j/2})$ with $\ell = (\ell_1,\ell_2)$, where
$\hat{s}_{j,\ell}$ is the orientation of $\Gamma$ in $Q_{j,\ell} \in \mathcal{Q}_j$ defined as in \eqref{eq:local_slope} with
$\ell = (\ell_1,\ell_2)$. Letting $\hat k_{j,s}(\ell) := k_{j,s}+2^{j/2}\hat{s}_{j,\ell}$ \gk{and} using the relation $k_{j,s} = 2^{j/2}s$,
we conclude that, for each $Q_{j,\ell} \in \mathcal{Q}_j$,
\[
\hat{k}_{j,s}(\ell) = 2^{j/2}\Bigl(s + E^{'}(0) + E^{''}(\xi_{\ell_2})\frac{\ell_2}{2^{j/2}}\Bigr).
\]
From this, for each $s = s(\lceil j_0/2 \rceil,q_0) \in \mathbb{S}_{J/2}$, in the case $j_0 \ge J/4$ we define
\begin{eqnarray*}
\wq{\tilde{\Lambda}^0_{J,s}} &=& \{\lambda = (-1,s,m,p) \,\,\text{or} \,\,(j,s,m,p) \in \tilde{\Lambda}_{J,s} : \mathrm{int}(\supp(\psi^{\sharp}_{\lambda})) \cap \Gamma \cap \mathrm{int}(Q_{j,\ell}) \neq \emptyset \,\, \text{and} \\
&& \quad |\hat{k}_{j,s}(\ell)| \leq 2^{\frac{J-j}{4}} \,\, \text{for} \,\, Q_{j,\ell} \in \mathcal{Q}_j, j = j_0,\dots,J\},
\end{eqnarray*}
and in the case $j_0 < J/4$, we set
\begin{eqnarray*}
\wq{\tilde{\Lambda}^0_{J,s}} &=&
\{\lambda = (j,s,m,p) \in \tilde{\Lambda}_{J,s} : j = -1,j_0,\dots,\frac{J}{4}-1\} \\ &\cup&
\{\lambda = (j,s,m,p) \in \tilde{\Lambda}_{J,s} : \mathrm{int}(\supp(\psi^{\sharp}_{\lambda})) \cap \Gamma \cap \mathrm{int}(Q_{j,\ell}) \neq \emptyset \,\, \text{and} \,\, |\hat{k}_{j,s}(\ell)| \leq 2^{\frac{J-j}{4}} \\
&& \quad \text{for} \,\, Q_{j,\ell} \in \mathcal{Q}_j, j = \frac{J}{4},\dots,J\}.
\end{eqnarray*}
\wq{This \gk{can be regarded as} the set of estimated indices for the $n_{J,s}$ largest shearlet coefficients\gk{. Hence}
sparsity patterns in the shearlet coefficients are encoded here.}

We next estimate the sizes of those sets. For this, note that, for $s \in \mathbb{S}_{J/2}$, $j \ge 0$, \wq{$p \in \N_0$} fixed and each
$Q_{j,\ell} \in \mathcal{Q}_j$, we have
\begin{equation}\label{eq:count01}
\sharp\bigl( \{\lambda = (j,s,m,p) \in \tilde\Lambda_{J,s} : \mathrm{int}(\supp(\psi^{\sharp}_{\lambda})) \cap \mathrm{int}(Q_{j,\ell}) \cap \Gamma \neq \emptyset \}  \bigr) \lesssim \wq{2^{2J\rho}}(1+|\hat{k}_{j,s}(\ell)|).
\end{equation}
\wq{To show estimate \eqref{eq:count01}, we first consider the simple case, namely when $\hat s_{j,\ell} = 0$ with $\ell = (0,0)$
and $s = 0$. This implies $\hat k_{j,s}(\ell) = 0$ in \eqref{eq:count01}. From Proposition \ref{prop:extra_shearlets}, we have
\begin{equation}\label{eq:normal_supp}
\text{int}(\text{supp}(\psi^{\sharp}_{\lambda})) \subset A^{-1}_{j}(A_{j\rho}[-L,L]^2+D_pm) \quad \text{for some}\,\, L>0.
\end{equation}
Also, since the discontinuity curve $x_1 = E(x_2)$ is $C^2$ smooth, there exists some positive constant $K>0$ such that
\[
\mathrm{int}(Q_{j,\ell}) \cap \Gamma \subset A^{-1}_j([0,K]^2).
\]
Therefore,
\begin{eqnarray*}
\text{LHS in \eqref{eq:count01}} &\leq& \sharp\bigl( \{m \in \Z^2 : A^{-1}_{j}(A_{j\rho}[-L,L]^2+D_pm) \cap A^{-1}_j([0,K]^2) \neq \emptyset \} \bigr) \\
&\leq& \sharp\bigl( \{m \in \Z^2 : (A_{j\rho}[-L,L]^2+D_pm) \cap [0,K]^2 \neq \emptyset \}\bigr) \\
&\lesssim& |A_{j\rho}|\cdot|D^{-1}_p| = 2^{\frac{3}{2}j\rho}2^{p}.
\end{eqnarray*}
This \gk{yields} \eqref{eq:count01} with $\hat k_{j,s}(\ell) = 0$, since $j \leq J$ and $p \leq \frac{J\rho}{2}$ for
$\lambda = (j,s,m,p) \in \tilde \Lambda_{J,s}$. For the general case, we note that \eqref{eq:normal_supp} is precisely
the form of standard shearlets considered in \cite{KL11} except for \gk{an} additional scaling matrix $A_{j\rho}$ and
oversampling matrix $D_p$\gk{. This} allows us to} use the same argument as in \cite{KL11}(page 19) to show \eqref{eq:count01}\gk{. And}
the  additional factor \wq{$2^{2J\rho}$} comes from the oversampling parameter $p$ associated with sampling matrix \wq{$D_p$ and $A_{j\rho}$
in \eqref{eq:normal_supp}.} 
In addition, we will use the observation that each monotonic function $h \in L^1([a,b])$ satisfies
\begin{equation}\label{eq:number_of_indices}
\sum_{\{\ell : \hat{k}_{j,s}(\ell) \in [a,b], Q_{j,\ell} \in \mathcal{Q}_j\}} h(\hat{k}_{j,s}(\ell)) \lesssim \frac{1}{\inf_{t \in [0,1]} |E^{''}(t)|}\int_{[a,b]}|h(x)|dx
\end{equation}
Now, let $C_{E} := \frac{1}{\inf_{t \in [0,1]} |E^{''}(t)|}$. Notice that this is a positive constant, since the discontinuity
curve $\Gamma$ given by $x_1 = E(x_2)$ is a $C^2$ smooth curve of non-vanishing curvature. By \eqref{eq:count01} and \eqref{eq:number_of_indices},
in the case $j_0 \ge J/4$, we then have
{\allowdisplaybreaks
\begin{eqnarray*}
\sharp(\wq{\tilde{\Lambda}^0_{J,s}}) &\lesssim& \sum_{j = j_0}^{J}\sum_{\{\ell : |\hat{k}_{j,s}(\ell)| \leq 2^{\frac{J-j}{4}}, Q_{j,\ell} \in \mathcal{Q}_j\}}\wq{2^{2J\rho}}(1+|\hat{k}_{j,s}(\ell)|) \\
&\lesssim& C_E\cdot\sum_{j = j_0}^{J} \wq{2^{2J\rho}}2^{\frac{J-j}{2}}\\
& \lesssim& C_E\cdot 2^{\frac{J-j_0}{2}}\wq{2^{2J\rho}}.
\end{eqnarray*}
}
Also, in the case $j_0 < J/4$, we can compute
{\allowdisplaybreaks
\begin{eqnarray*}
\sharp(\wq{\tilde{\Lambda}^0_{J,s}}) &\lesssim&  \sum_{j = j_0}^{J/4-1}\sharp(\{\lambda = (j,s,m,p)\in \tilde{\Lambda}_{J,s} : p \leq \frac{J\rho}{2}\})
+ \sum_{j = J/4}^{J}\sum_{\{\ell : |\hat{k}_{j,s}(\ell)| \leq 2^{\frac{J-j}{4}}, Q_{j,\ell} \in \mathcal{Q}_j\}} \hspace*{-1cm} 2^{\frac{J\rho}{2}}(1+|\hat{k}_{j,s}(\ell)|) \\
&\lesssim& \sum_{j = j_0}^{J/4-1} \wq{2^{2J\rho}}2^{\frac{3}{2}j} + C_E \cdot \sum_{j = J/4}^{J} \wq{2^{2J\rho}}2^{\frac{J-j}{2}} \\
&\lesssim& C_E \cdot 2^{\frac{3}{8}J}\wq{2^{2J\rho}}\\
& \leq& C_E \cdot 2^{\frac{J-j_0}{2}}\wq{2^{2J\rho}}.
\end{eqnarray*}
}
Proposition \ref{prop:cs_aniso_wavelet} then implies
\begin{equation}\label{eq:sub_error}
\|{\bf c_{J,s}} - {\bf\hat{c}_{J,s}}\|^2_2 \lesssim (J^{-1})\wq{2^{-2J\rho}}2^{\frac{-J+j_0}{2}}
\Bigl(\sum_{\lambda \in \tilde{\Lambda}_{J,s} \cap(\wq{\tilde{\Lambda}^0_{J,s}})^c}|\langle f,\psi_{\lambda}\rangle|\Bigr)^2.
\end{equation}
\wq{with $n_{J,s} \sim J2^{\frac{J-j_0}{2}}2^{2J\rho}$ and $m_{J,s} \sim 2^{\frac{J-j_0}{2}}2^{3J\rho}$.}
Notice that we here used the fact that ${\bf c_{J,s}} = (\langle f,\psi_{\lambda}\rangle)_{\lambda \in \tilde{\Lambda}_{J,s}}$.

In a second step, we now estimate the error $\|f-\hat{f}\|_2$, thereby finishing the proof. We begin by estimating this
error by the three terms
\begin{eqnarray} \nonumber
\|f-\hat{f}\|^2_2 &\lesssim& \sum_{j_0 = 0}^{J}\sum_{s \in \{s(\lceil j^{'}/2 \rceil,q^{'}) \in \mathbb{S}_{J/2} : j^{'} = j_0\}}
\hspace*{-1cm} \|{\bf c_{J,s}}-{\bf \hat{c}_{J,s}}\|^2_2 + \sum_{j \ge J} \sum_{\lambda \in \Lambda_j}|\langle f,\psi_{\lambda}\rangle|^2 + \sum_{j = 0}^{J} \sum_{\lambda \in \Lambda_j \cap (\Lambda^0_j)^c}
|\langle f,\psi_{\lambda}\rangle|^2 \\ \label{eq:est100}
&=:& \mathrm{(I)} + \mathrm{(II)} + \mathrm{(III)},
\end{eqnarray}
where
\[
\Lambda_j = \{\lambda = (j^{'},s,m,p) \in \Lambda : j^{'} = j\} \mbox{ and }
\Lambda^0_j = \{\lambda = (j,s,m,p) \in \Lambda_j : p \leq \max(\frac{j\rho}{2},\frac{J\rho}{2})\}.
\]

We now turn to analyzing $\mathrm{(I)}, \mathrm{(II)}$, and  $\mathrm{(III)}$. Starting with $\mathrm{(I)}$, we let
$s = s(\lceil j_0/2 \rceil,q_0) \in \mathbb{S}_{J/2}$ and consider two cases:
\\
{\it Case 1}: $j_0 \ge J/4$. First,
\begin{eqnarray*}
\lefteqn{{\tilde{\Lambda}_{J,s}} \cap (\wq{\tilde{\Lambda}^0_{J,s}})^c}\\  &=& \{\lambda = (-1,s,m,p) \,\, \text{or} \,\,(j,s,m,p) \in \tilde{\Lambda}_{J,s} : \mathrm{int}(\supp(\psi^{\sharp}_{\lambda})) \cap \Gamma \cap \mathrm{int}(Q_{j,\ell}) \neq \emptyset \,\, \text{and} \\
&& \text{for} \,\, |\hat{k}_{j,s}(\ell)| > 2^{\frac{J-j}{4}} \,\, Q_{j,\ell} \in \mathcal{Q}_j, j = j_0,\dots,J\} \cup \{\lambda \in \tilde{\Lambda}_{J,s}  : \mathrm{int}(\supp(\psi^{\sharp}_{\lambda})) \cap \Gamma = \emptyset \} \\ &=:& \mathcal{I}_0 \cup \mathcal{I}_1.
\end{eqnarray*}
By \eqref{eq:sub_error}, we obtain
\begin{equation}\label{eq:com01}
\|{\bf c_{J,s}}-{\bf \hat{c}_{J,s}}\|^2_2 \lesssim (J^{-1})2^{\frac{j_0}{2}}2^{-\frac{J}{2}}\wq{2^{-2J\rho}}\Bigl( \sum_{\lambda \in \mathcal{I}_0} |\langle f,\psi_{\lambda}\rangle| + \sum_{\lambda \in \mathcal{I}_1} |\langle f,\psi_{\lambda}\rangle|\Bigr)^2.
\end{equation}
Setting
\[
\mathcal{L}_j := \{\ell : |\hat{k}_{j,s}(\ell)|>2^{\frac{J-j}{4}}, Q_{j,\ell} \in \mathcal{Q}_j\},
\]
by Proposition \ref{prop:coeff_decay}(i), \eqref{eq:count01} and \eqref{eq:number_of_indices}, for sufficiently large $J>0$ we have
{\allowdisplaybreaks
\begin{eqnarray}\label{eq:com02}
2^{\frac{j_0}{2}}2^{-\frac{J}{2}}\wq{2^{-2J\rho}}\Bigl( \sum_{\lambda \in \mathcal{I}_0} |\langle f,\psi_{\lambda}\rangle| \Bigl)^2 &\lesssim&
2^{\frac{j_0}{2}}2^{-\frac{J}{2}}\wq{2^{-2J\rho}}\Bigl( \sum_{j = j_0}^{J}\wq{2^{2J\rho}}\sum_{\ell \in \mathcal{L}_j }(1+|\hat{k}_{j,s}(\ell)|)\frac{2^{-\frac{3}{4}j}2^{3\rho j}}{|\hat{k}_{j,s}(\ell)|^3}\Bigl)^2 \nonumber \\
&\lesssim& C_E\cdot2^{\frac{j_0}{2}}2^{-\frac{J}{2}}\wq{2^{-2J\rho}}\Bigl( \sum_{j = j_0}^{J}
\wq{2^{2J\rho}}2^{-\frac{3}{4}j}2^{3\rho j}2^{-\frac{J-j}{4}}\Bigr)^2 \nonumber \\
&\lesssim& 2^{-\frac{j_0}{2}}2^{-J}\wq{2^{8J\rho}}.
\end{eqnarray}}
Using that, for each $j \in \{j_0-1,\dots,J\}$,
\[
\sharp(\{ \lambda = (j^{'},s,m,p) \in \mathcal{I}_1 : j^{'} = j \}) \lesssim 2^{\frac{3}{2}j+\frac{J\rho}{2}},
\]
Proposition \ref{prop:coeff_decay}(ii) implies that
\begin{eqnarray}\label{eq:com03}
2^{\frac{j_0}{2}}2^{-\frac{J}{2}}\wq{2^{-2J\rho}}\Bigl( \sum_{\lambda \in \mathcal{I}_1} |\langle f,\psi_{\lambda}\rangle| \Bigl)^2 &\lesssim& 2^{\frac{j_0}{2}}2^{-\frac{J}{2}}\wq{2^{-2J\rho}}\Bigl( \sum_{j = j_0}^{J}
2^{\frac{3}{2}j+\frac{J\rho}{2}}(2^{4\rho j}2^{-3j}2^{\frac{J\rho}{4}})\Bigr)^2 \nonumber \\
&\lesssim& 2^{-\frac{j_0}{2}}2^{-J}\wq{2^{8J\rho}}.
\end{eqnarray}
By \eqref{eq:com01}, \eqref{eq:com02} and \eqref{eq:com03}, we have
\beq \label{eq:est10}
\|{\bf c}_{J,s} - {\bf \hat{c}_{J,s}}\|^2_2 \lesssim (J^{-1})2^{-\frac{j_0}{2}}2^{-J}\wq{2^{8J\rho}}.
\eeq
\\
{\it Case 2}: $j_0 < J/4$. First, in analogy with the previous case, we observe that
\begin{eqnarray*}
\lefteqn{{\tilde{\Lambda}_{J,s}} \cap(\wq{\tilde{\Lambda}^0_{J,s}})^c}\\ &=& \{\lambda = (j,s,m,p) \in \tilde{\Lambda}_{J,s} : \mathrm{int}(\supp(\psi^{\sharp}_{\lambda})) \cap \Gamma \cap \mathrm{int}(Q_{j,\ell}) \neq \emptyset \,\, \text{and} \,\, |\hat{k}_{j,s}(\ell)| > 2^{\frac{J-j}{4}} \\
&& \quad \text{for} \,\, Q_{j,\ell} \in \mathcal{Q}_j, j = J/4,\dots,J\} \\ &\cup& \{\lambda = (j,s,m,p) \in \tilde{\Lambda}_{J,s}  : \mathrm{int}(\supp(\psi^{\sharp}_{\lambda}))
\cap \Gamma = \emptyset, j = J/4,\dots,J \} \\
&=:& \tilde{\mathcal{I}}_1 \cup \tilde{\mathcal{I}}_2.
\end{eqnarray*}
Hence,
\[
\|{\bf c}_{J,s}-{\bf \hat{c}_{J,s}}\|_2^2 \lesssim (J^{-1})2^{\frac{j_0}{2}}2^{-\frac{J}{2}}\wq{2^{-2J\rho}}\Bigl(
\sum_{\lambda \in \tilde{\mathcal{I}}_1} |\langle f,\psi_{\lambda}\rangle| + \sum_{\lambda \in \tilde{\mathcal{I}}_2} |\langle f,\psi_{\lambda}\rangle|\Bigr)^2
\]
For each term $\sum_{\lambda \in \tilde{\mathcal{I}}_r} |\langle f,\psi_{\lambda}\rangle|$ with $r = 1,2$, we next apply the same argument
as in {\it Case 1} except for replacing $j_0$ by $J/4$, which similarly yields
\beq \label{eq:est11}
\|{\bf c}_{J,s} - {\bf \hat{c}_{J,s}}\|^2_2 \lesssim (J^{-1})2^{-\frac{j_0}{2}}2^{-J}2^{\wq{8J\rho}}.
\eeq
This shows that this estimate actually holds for each $s \in \mathbb{S}_{J/2}$.

Concluding this part, since
\[
\sharp(\{s(\lceil j^{'}/2 \rceil,q^{'}) \in \mathbb{S}_{J/2} : j^{'} = j_0\}) \lesssim 2^{j_0/2},
\]
\eqref{eq:est10} and \eqref{eq:est11} imply
\beq \label{eq:est13}
\mathrm{(I)} \lesssim \sum_{j_0 = 0}^{J} (2^{\frac{j_0}{2}})(J^{-1})(2^{-\frac{j_0}{2}}2^{-J}2^{\frac{13J\rho}{2}}) \lesssim  2^{-J(1-\wq{8\rho})}.
\eeq

Also, (II) can be estimated to show
\wq{
\begin{equation}\label{eq:estII}
\sum_{j \ge J}\sum_{\lambda \in \Lambda_j} |\langle f,\psi_{\lambda}\rangle|^2 \lesssim 2^{-J(1-8\rho)}
\end{equation}
in the same way as} in the proof of Theorem \ref{thm:sparsity} in \cite{KL15}. 
\wq{For (III), note that there \gk{exist} about $2^{2j+p}$ shearlets $\psi_{\lambda}$ with $\lambda = (j,s,m,p) \in \Lambda$ for each
fixed $j$ and $p \ge 0$, since $|s| \lesssim 2^{j/2}$\gk{. Moreover,} there \gk{exist} about $|A_j| \cdot |D^{-1}_p| = 2^{\frac{3}{2}j}2^{p}$ translates
for each shear parameter $s$. In particular, this implies
\begin{equation}\label{eq:p_count}
\sharp\big( \{\lambda = (j,s,m,p^{'}) \in \Lambda_j : p^{'} = p \}\big) \lesssim 2^{2j+p}.
\end{equation}
Here, we only need to count the number of shearlets $\psi_{\lambda}$ with $\text{int}(\text{supp}(\psi_{\lambda}))\cap \text{int}(\text{supp}(f)) \neq \emptyset$. Now using Proposition \ref{prop:control_p} and \eqref{eq:p_count}, we \gk{obtain}
\begin{eqnarray}\label{eq:estIII}
\sum_{j = 0}^{J}\sum_{\lambda \in \Lambda_j \cap (\Lambda^0_j)^c}|\langle f,\psi_{\lambda}\rangle|^2 \nonumber &\lesssim& \sum_{j=0}^{J}\sum_{p > \frac{J}{2}\rho}2^{2j+p}2^{-\alpha p} \nonumber \\
&\lesssim& 2^{-(\frac{(\alpha-1)\rho}{2}-2)J} \leq 2^{-J}.
\end{eqnarray}
For the last inequality, we used $\alpha \ge \frac{6}{\rho}+1$.}

Inserting the estimates for $\mathrm{(I)}, \mathrm{(II)}$, and  $\mathrm{(III)}$, namely \wq{\eqref{eq:est13}, \eqref{eq:estII}, and \eqref{eq:estIII},} respectively, into
\eqref{eq:est100} \wq{and using Lemma \ref{lemm:deltaJ}} completes the proof. \hfill \qed

\section{Numerical Results}\label{sec:numerics}

In this section, we provide numerical results for our directional sampling scheme described in \gk{Subs}ection \ref{subsec:scheme1}. 
\gk{We first provide details on the implementation of our sampling-reconstruction scheme, which we will in this section also
refer to as {\em directional sampling-reconstruction scheme}, or, if only the sampling part is meant, as {\em directional 
sampling scheme}. This part will be followed be numerical experiments on comparison of our scheme with other previously
developed schemes.}

\subsection{\gk{Implementation of our Sampling-Reconstruction Scheme}}

\wq{We start by recalling the notion of the discrete Fourier transform $\mathcal{F}$ of a sequence $\{a(n)\}_{n \in \Z^d}$, which is defined by
\[
\mathcal{F}(a) (\xi) = \sum_{n \in \Z^d} e^{-2\pi i n \cdot \xi}.
\]

\subsubsection{\gk{Discretization of the Filters $G^d_s$}}

We first describe the discretization of the filters $\tilde{G}^d_s$ from \eqref{eq:dfilters} and their duals. In \cite{KLR14}, we developed digital 
shearlet filters $\psi^{d}_j$ and $\phi^d$, which can be used to discretize $\hat g(A^{-1}_j \xi)$ and $\hat \varphi^0(\xi)$ in \eqref{eq:dfilters} 
by taking the discrete Fourier transform $\mathcal{F}$. \gk{Utilizing those, we obtain} a discrete version of $\hat G_0$ defined in \eqref{eq:dfilters} 
given as
\[
\mathcal{F}(G^d_0)(\xi) = |\mathcal{F}(\phi^d)(\xi)|^2 + \sum_{j = 0}^{J}|\mathcal{F}(\psi^d_j)(\xi)|^2.
\]
\gk{Notice that we only take scales $j$ up to $j = J$ with some finite scale $J$ for a given $N \times N$ digital image $u$ with $N \sim 2^J$. Later
we will choose} $J = 2$ and $4$ so that \gk{the} shearing parameters $s \in \mathbb{S}_{J/2}$ for the directional filters $G_s$ are given as 
$s = -1/2,0,1/2$ and $s=-3/4, \dots, 3/4$, respectively. 

\gk{Continuing,} for each $s \in \mathbb{S}_{J/2}$, we \gk{then} apply a linear operator $S^d_s$ faithfully discretizing $S_s$ to obtain \gk{the} 
digital directional filters
\[
\mathcal{F}(G^d_s)(\xi) = \mathcal{F}(S^d_s(G^d_0))(\xi),
\]
which now discretizes $\hat G_s (\xi)$ in \eqref{eq:dfilters}. We refer to \cite{KLR14} for more details on the digital shear operator $S^d_s$. 
Note that, for all $s \in \mathbb{S}_{J/2}$, we have fixed $j_0$ in \eqref{eq:dfilters} as $j_0 = 0$ for our directional filters $G^d_s$ to simplify 
our implementation. \gk{For an illustration of $\mathcal{F}({G}^d_s)$, we refer to Figure \ref{fig:directional_filter}.}

\begin{figure}[htb]
\begin{center}
\hspace*{\fill}
\includegraphics[width=0.3\textwidth]{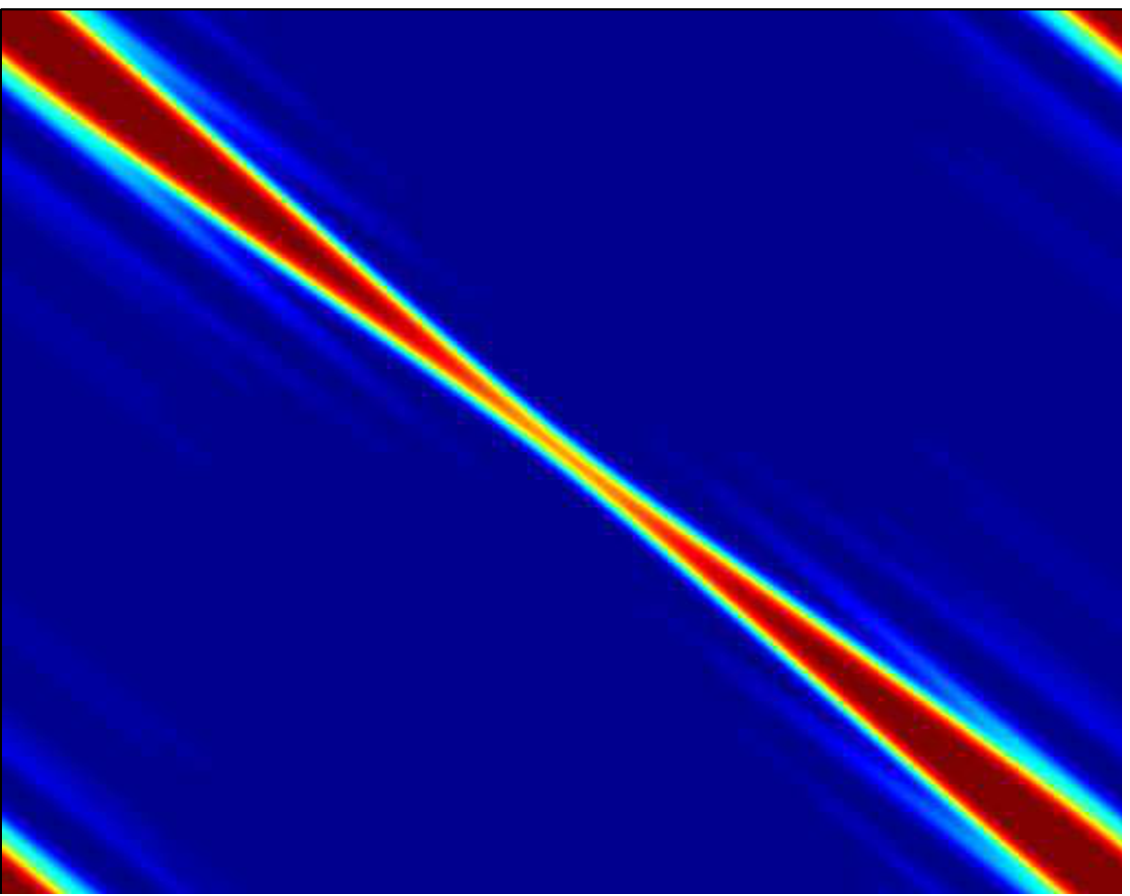}
\hfill
\includegraphics[width=0.3\textwidth]{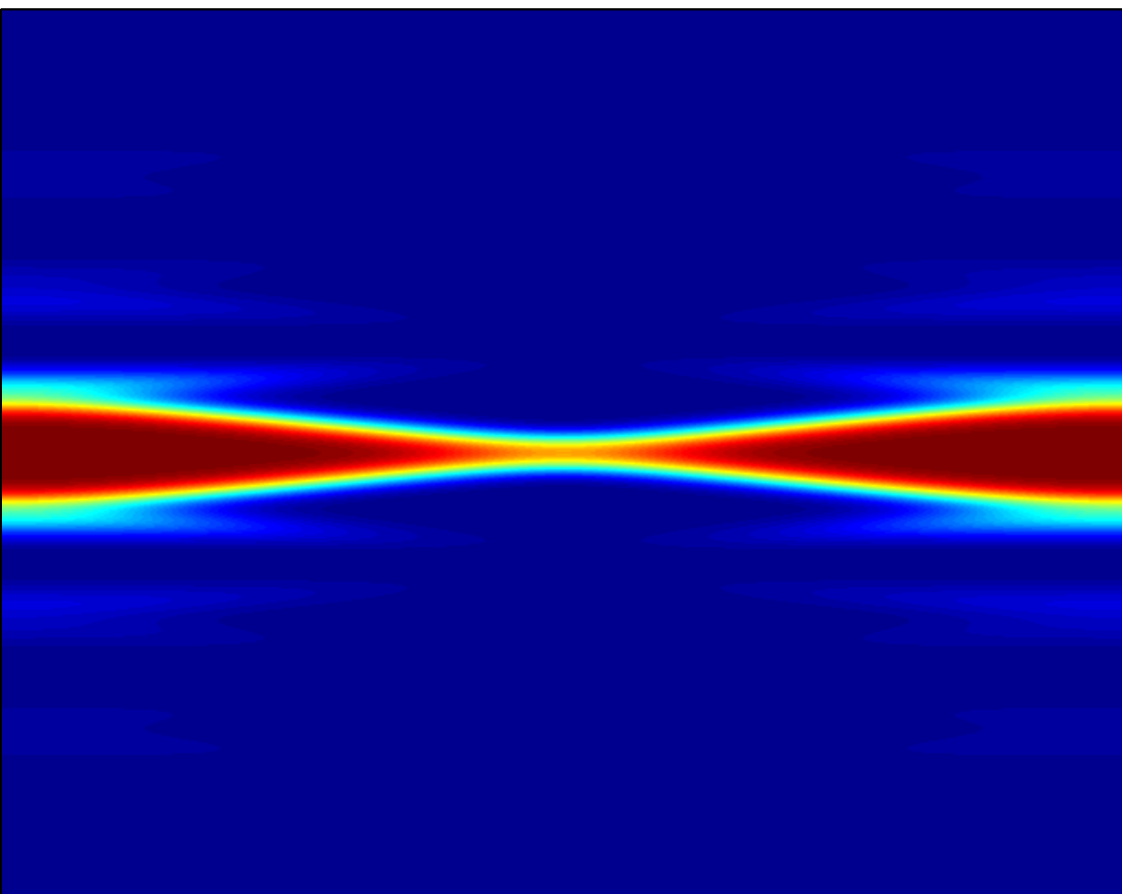}
\hspace*{\fill}
\put(-339,-13){(a)}
\put(-135,-13){(b)}
\end{center}
\caption{(a) 2D plot of  $\mathcal{F}({G}^d_s)$ with $s = 1$. (b) $\mathcal{F}(G^d_s)$ with  $s = 0$.}
\label{fig:directional_filter}
\end{figure}

In addition to those filters, we also need directional filters for $G_s \circ \wq{R}$. These filters are obtained by simply switching the variables. 
To simplify our presentation, we only consider $G^d_s$ associated with $A_j$ and $S_s$.  

We \gk{can} now define dual filters $\tilde{G}^d_s$ for $G^d_s$ by
\[
\mathcal{F}(\tilde{G}^d_s) = \frac{\overline{\mathcal{F}(G^d_s)}}{\sum_{s \in \mathcal{S}_{J/2}}|\mathcal{F}(G^d_s)|^2},
\]
where \gk{the bar denotes the} complex conjugate. \gk{As in \cite{KLR14},} one can ensure $\sum_{s \in \mathcal{S}_{J/2}}|\mathcal{F}(G^d_s)|^2 > C$ 
for some $C>0$ with a suitable choice for digital shearlet filters $\psi_j$. From this, we have the following reconstruction formula:
\[
u = \sum_{s \in \mathcal{S}_{J/2}}\tilde{G}^d_s \star G^d_s \star u
\]
where $\star$ is a discrete circular convolution.

\subsubsection{\gk{Discretization of our Sampling Scheme $\Delta_J$}}

For the probability density function $p_{J,s}$, \gk{in the implementation} we use
\begin{equation}\label{eq:dprob}
p_{J,s}(n) = \frac{c_s}{(1+|n_1|)^5(1+|2^{J/2}n_2-sn_1|)^5}
\end{equation}
for $s \in \mathbb{S}_{J/2}$ rather than \gk{the} one defined in \eqref{eq:probability} to incorporate the sparsity and anisotropy of
shearlets. \gk{The reason being that we} found this choice \gk{to in fact provide even} better performance.
From this, we define $\Delta_{J,s}$ by the set of Fourier sampling points randomly drawn according to the probability density function \eqref{eq:dprob} and let
\[
\Delta_J = \bigcup_{s \in \mathbb{S}_{J/2}} \Delta_{J,s}.
\]
\gk{Notice that f}or $\Delta_{J,s}$, we simply take the same number of sampling points for each $s \in \mathbb{S}_{J/2}$, and the total number of 
Fourier samples \gk{to be} given as
\[
\sharp(\Delta_J) = \sum_{s \in \mathbb{S}_{J/2}} \sharp(\Delta_{J,s}).
\]
With this set $\Delta_J$, we can now define our random mask operator $P_{\Delta_J}$ such that $P_{\Delta_J}(u)(n)$ \gk{equals} $u(n)$ if $n \in \Delta_J$ 
and $0$ otherwise. \gk{The Fourier sampling points $\Delta_J$ obtained by our directional sampling scheme, are illustrated in Figure \ref{fig:sampling_mask}(a).}

\begin{figure}[htb]
\begin{center}
\includegraphics[width=0.4\textwidth]{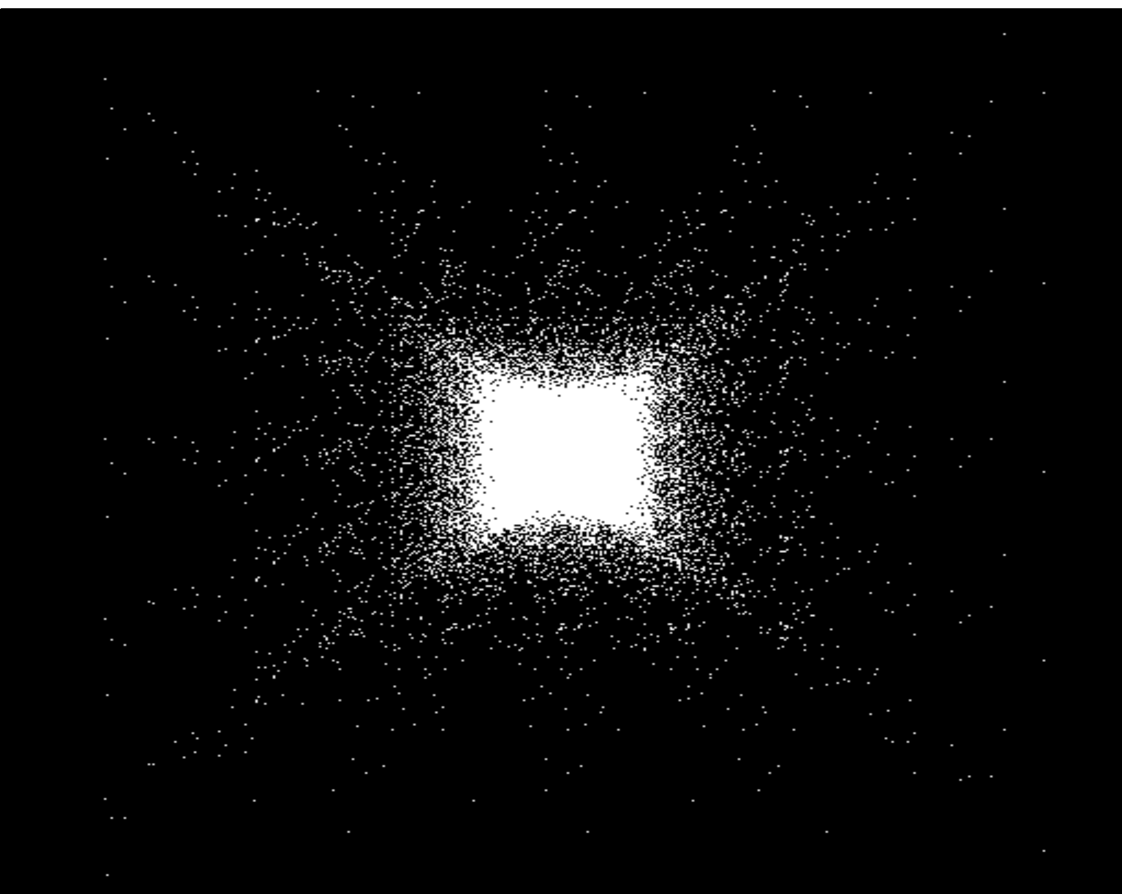}
\hspace*{0.5cm}
\includegraphics[width=0.4\textwidth]{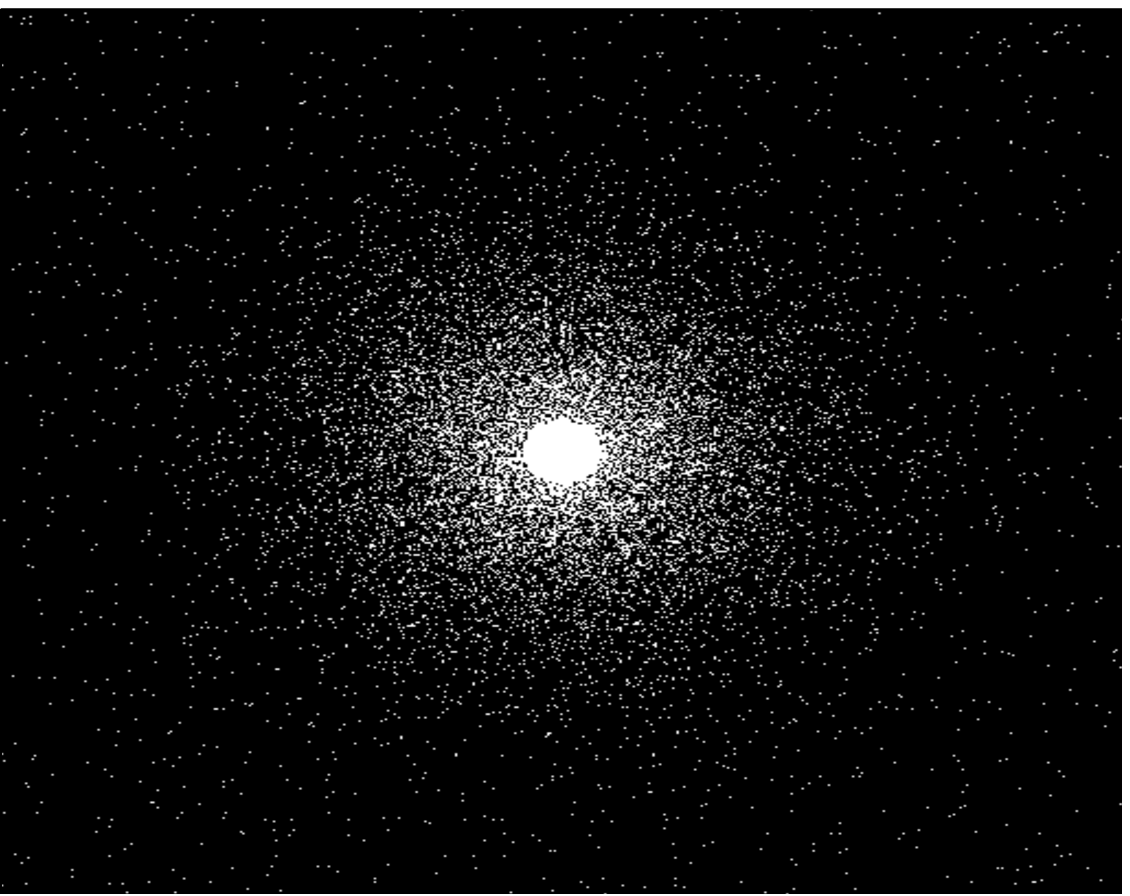}
\put(-310,-13){(a)}
\put(-101,-13){(b)}
\end{center}
\caption{Subsampling maps for Fourier measurements with 5\% subsampling with sampling set $\Delta_J$:
(a) Subsampling map for our directional sampling scheme -- for this, we take a union of all directional subsampling maps \gk{$\Delta_{J,s}$}.
(b) Subsampling map obtained with variable density sampling for wavelets -- see \cite{LDSP08} and \cite{KW13}.}
\label{fig:sampling_mask}
\end{figure}

\subsubsection{\gk{Discretization of the Reconstruction Scheme $\mathcal{R}(f,\Delta_J)$}}

Having defined the directional filters $G^d_s$, we now \gk{turn to} derive \gk{a discretization} for implementing \gk{our reconstruction
scheme \eqref{eq:recscheme} and \eqref{eq:l1_min}}. For this, we first \gk{let} $\mathcal{W}_J$ \gk{be} the discrete wavelet transform 
associated with the anisotropic scale matrix $A_j$ up to the finest scale $j = J$. For more details on $\mathcal{W}_J$, we refer to 
\cite{L10}. 

\gk{Next, note} that $P_{\Psi_{J,s}}$ in \eqref{eq:l1_min} is given as \eqref{eq:projection}, where \gk{the} basis elements $\sigma_{\lambda}$ are 
obtained by applying the shear operator $S_s$ for anisotropic wavelets associated with $A_j$ for each $s \in \mathbb{S}$. Therefore, for 
a given finite sequence of coefficients $c = (c_{\lambda})_{\lambda \in \tilde{\Lambda}_{J,s}}$, the wavelet expansion $\sum_{\lambda \in 
\tilde{\Lambda}_{J,s}} c_{\lambda} \sigma_{\lambda}$ can be discretized by $S^d_{s}(\mathcal{W}^*_J(c))$, where $\mathcal{W}^*_J$ is the 
conjugate transpose of $\mathcal{W}_J$ and $\mathcal{W}_J\mathcal{W}^*_J = \mathcal{W}^*_J\mathcal{W}_J = I$.

We are now ready to derive our discrete implementation for \eqref{eq:l1_min} as follows. For this, we let $u$ be a target image we want to 
recover from Fourier samples given by $y = P_{\Delta_J}\mathcal{F}(u)$. \gk{This is then achieved by first applying the discretization of 
\eqref{eq:l1_min} given by}
\begin{equation}\label{eq:implementation}
\hat{c}_s = \argmin\|c_s\|_1
\gk{\quad \text{subject to} \quad} \mathcal{F}(G^d_s) \odot y = P_{\Delta_J}\mathcal{F}(S^d_s(\mathcal{W}^*_J(c_s)))
\end{equation}
for each $s \in \mathbb{S}_{J/2}$, where $\odot$ is entry-wise multiplication between two vectors. After obtaining coefficient vectors $\gk{\hat{c}_s}$ 
from \eqref{eq:implementation}, \gk{we then combine those coefficient sequences by the discrete version of \eqref{eq:recscheme} given as}
\[
\hat u = \sum_{s \in \mathbb{S}_{J/2}} \tilde{G}^d_s \star S^d_s(\mathcal{W}^*_J(\hat{c}_s)),
\]
\gk{to obtain the reconstructed image $\hat u$.}}

\subsection{\gk{Comparison with other Schemes}}

We now compare our directional sampling scheme with the wavelet based sampling scheme in \cite{LDSP08} as well as
a shearlet based sampling scheme\gk{. For the second scheme,} we take a \gk{standard} shearlet system as implemented in {\tt ShearLab} \cite{KLR14}
with the same Fourier measurements we used for our directional sampling scheme, i.e., we solve
\[
\wq{\min_{g} \|\Psi_J g\|_1 \gk{\quad \text{subject to} \quad} P_{\Delta_J}(\mathcal{F}(u-g)) = 0,}
\]
where $\Psi_J$ is the shearlet transform up to scale $J$. To be more precise, we
consider the following sampling schemes:

\begin{itemize}
\item  {\bf shear08}: Our directional sampling scheme with 8 directional filters. The redundancy of the associated shearlet system is 8.
\item {\bf shear16}: Our directional sampling scheme with 16 directional filters. The redundancy of the associated shearlet system is 16.
\item {\bf shear}: Shearlet based sampling scheme with a usual shearlet system and the same subsamplng map as one used for our directional sampling scheme.
For this, we used a shearlet system with 4 4 8 8 directions across scales, which leads to \gk{a} redundancy of 25.
\item {\bf wave01}: Wavelet based sampling scheme with a subsampling map obtained by variable density sampling -- see \cite{LDSP08} and \cite{KW13}.
\item {\bf wave02}: Wavelet based sampling scheme with the same subsampling map as \gk{the} one used for our directional sampling scheme.
\end{itemize}
All schemes are implemented in MATLAB and tested on a CPU 2.7GHz with 4.00GB memory. Figure \ref{fig:sampling_mask} illustrates
each of \gk{the} subsampling maps for directional sampling and variable density sampling. The directionality becomes very evident in
Figure \ref{fig:sampling_mask}(a) \gk{as compared to (b).}

Figure \ref{fig:experiments} shows PSNR values and running times for each of the sampling schemes for two test images.
It is evident that our directional sampling schemes {\bf shear08} and {\bf shear16} consistently outperform the wavelet
based sampling schemes {\bf wave01} and {\bf wave02} in terms of PSNR, while the running time of one of our directional
sampling schemes, {\bf shear08} is comparable with wavelet based sampling schemes. Moreover, with respect to PNSR
our new sampling scheme {\bf shear16} slightly outperforms the sampling scheme using the \gk{standard} shearlets {\bf shear},
but its running time is significantly faster, in fact about 4\gk{--}10 times faster depending on sampling rate.

\begin{figure}[h]
\begin{center}
\includegraphics[width=0.35\textwidth]{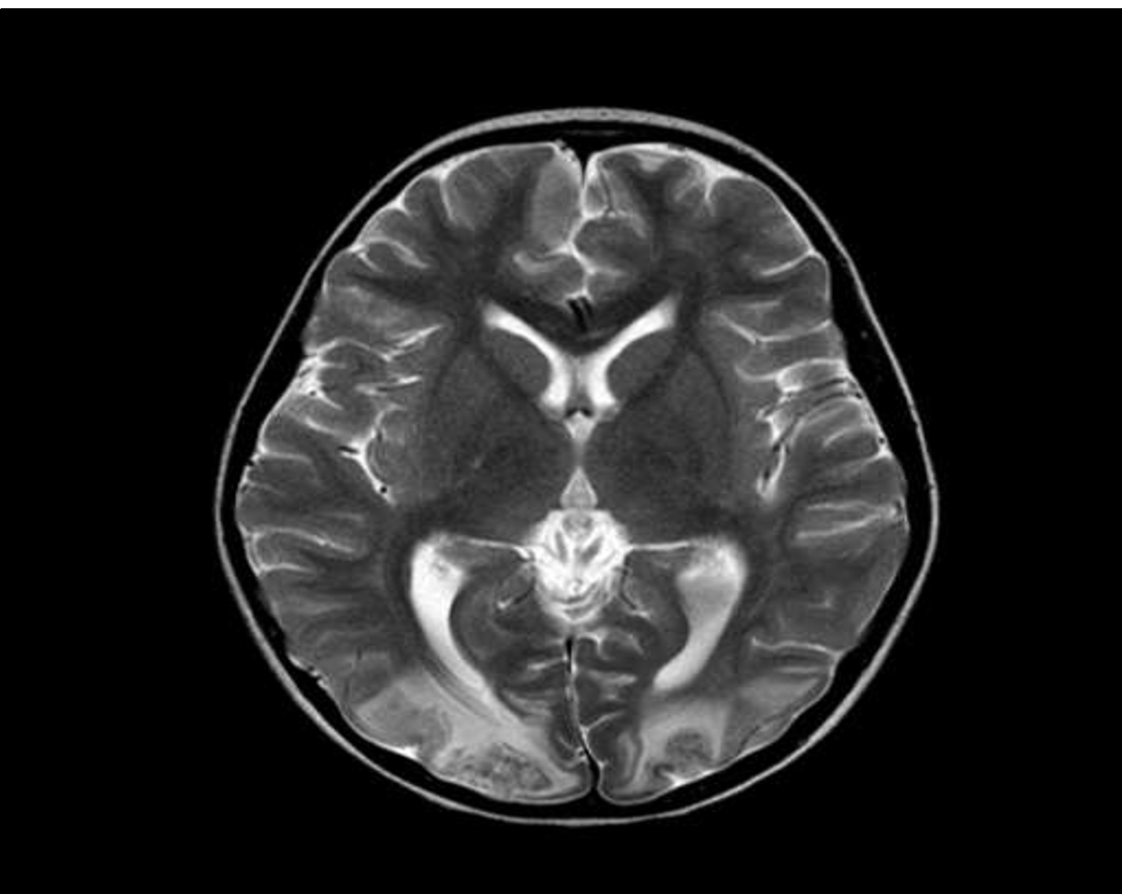}
\hspace*{0.5cm}
\includegraphics[width=0.35\textwidth]{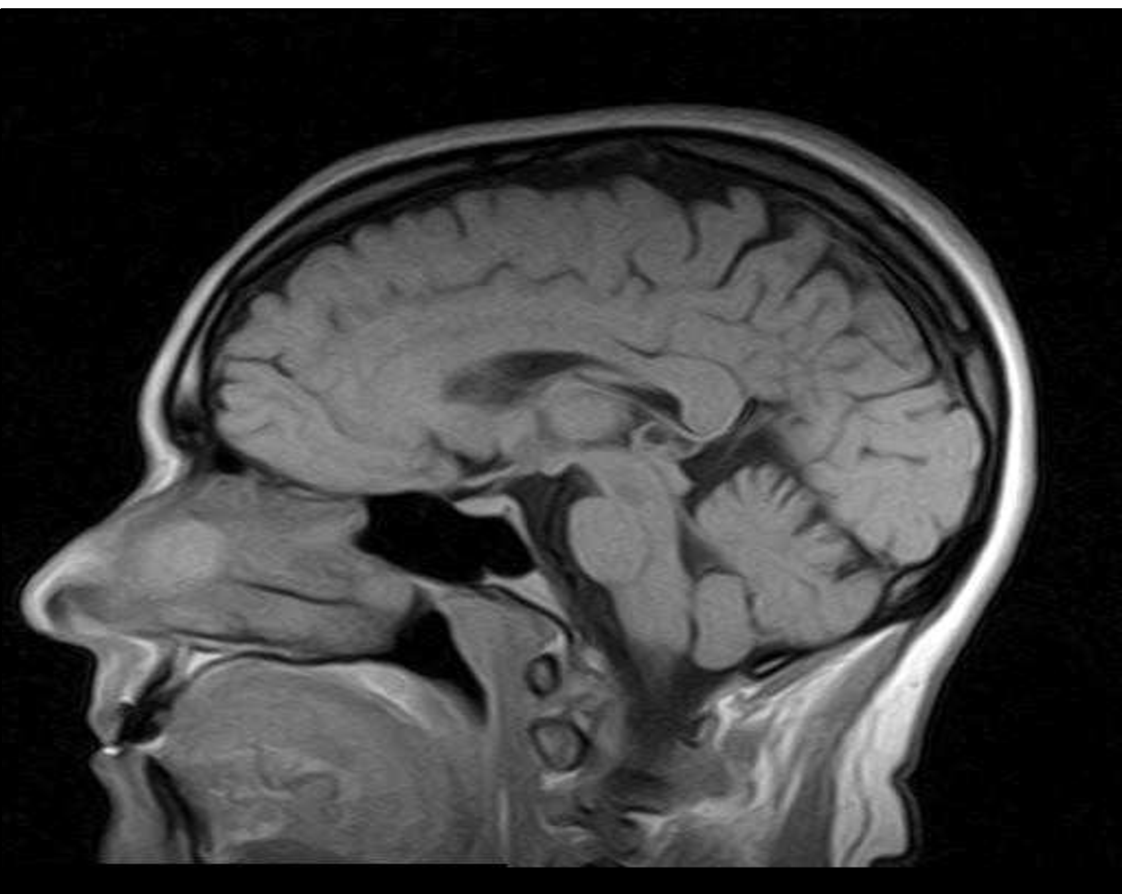}
\put(-273,-13){(a)}
\put(-85,-13){(b)}\\[3,5ex]
\includegraphics[width=0.4\textwidth]{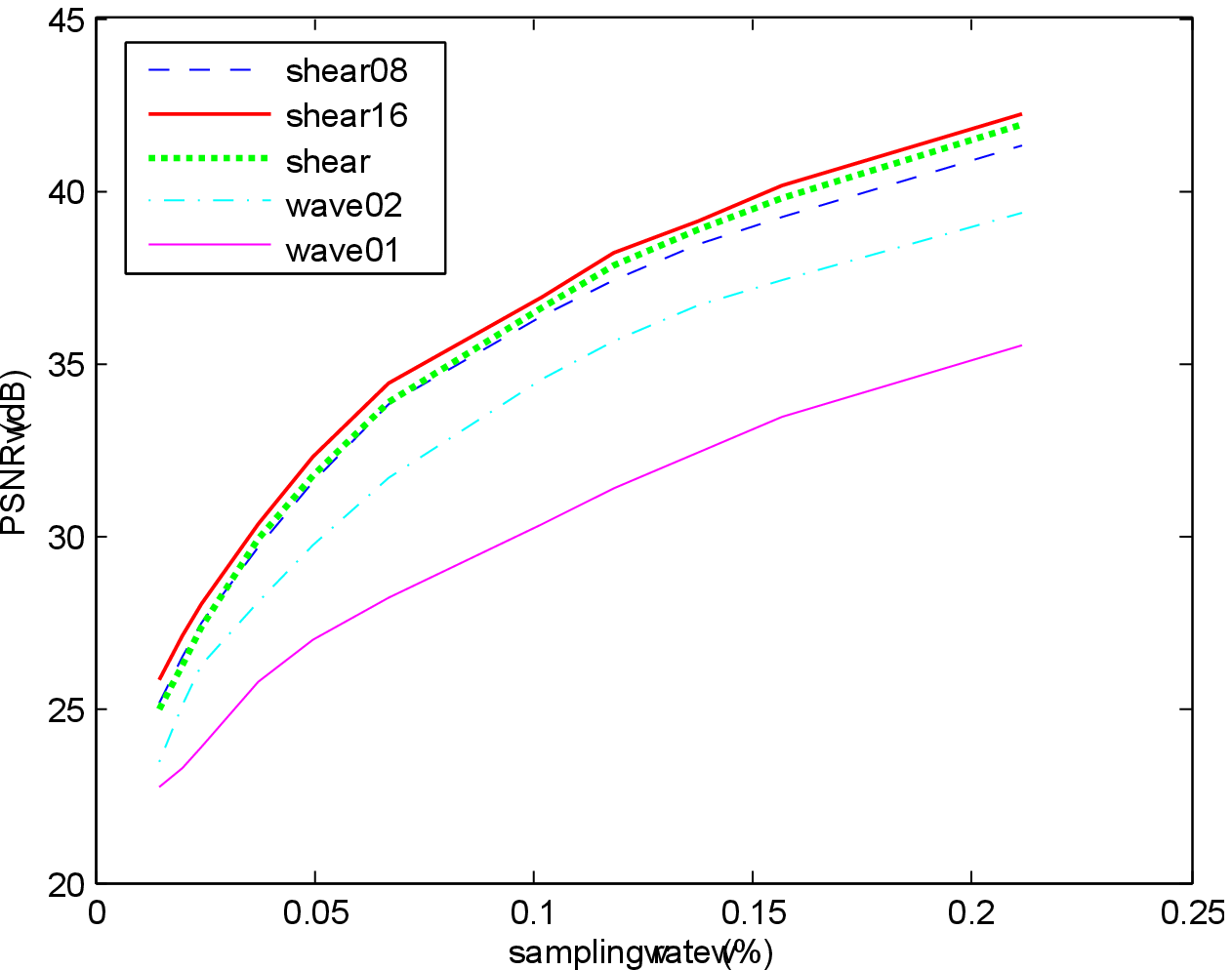}
\includegraphics[width=0.4\textwidth]{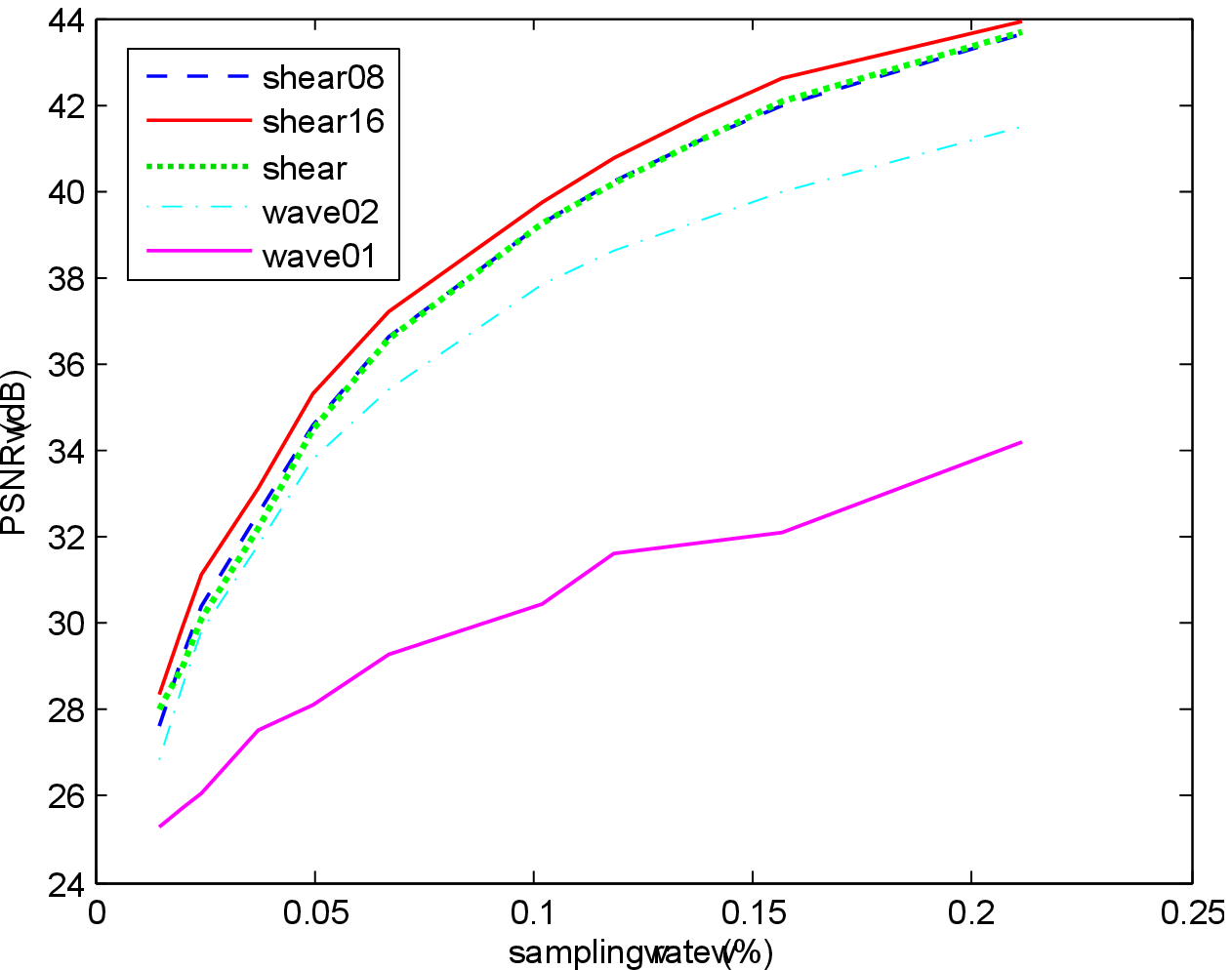}
\put(-290,-13){(c)}
\put(-95,-13){(d)}\\[2,5ex]
\includegraphics[width=0.4\textwidth]{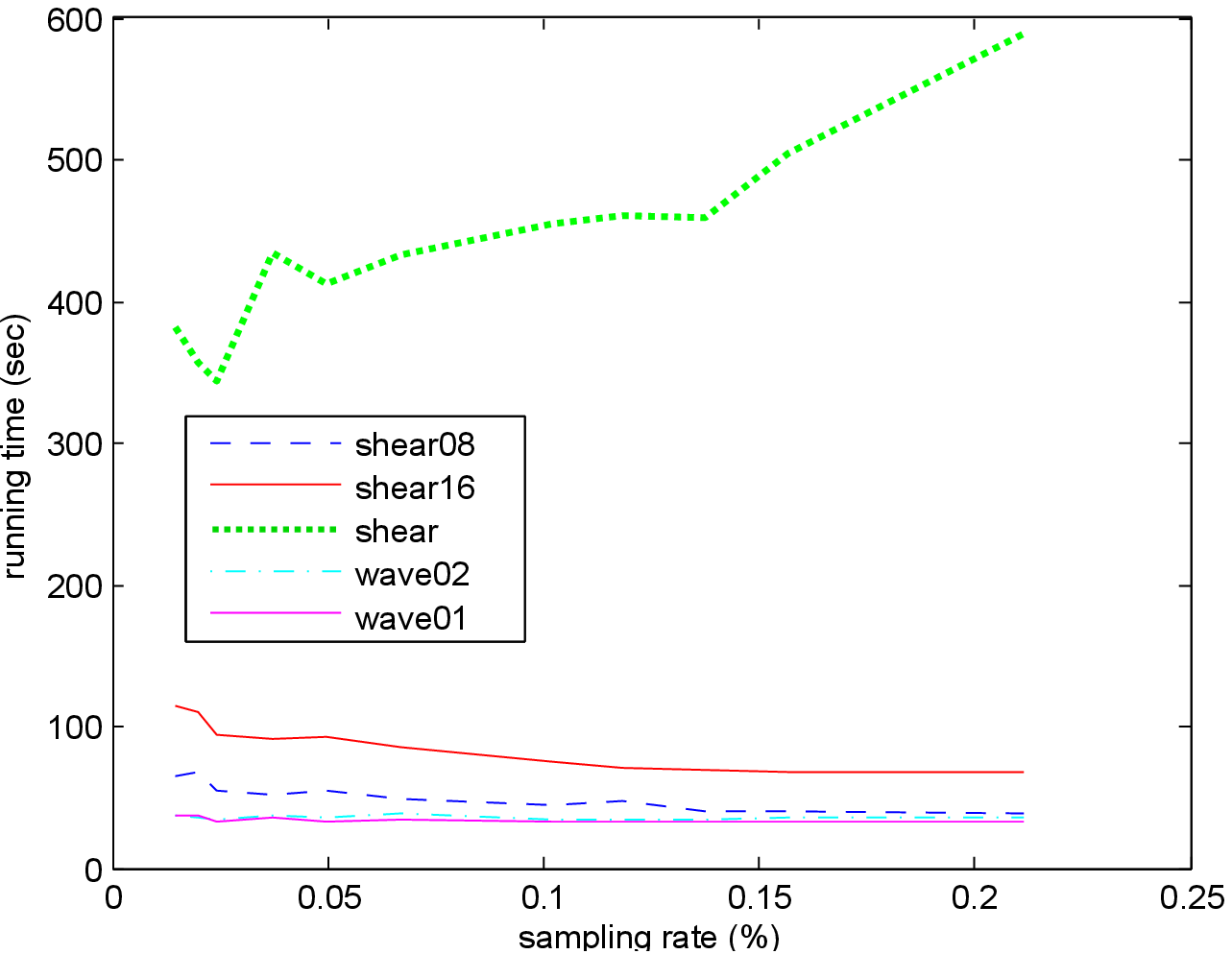}
\includegraphics[width=0.4\textwidth]{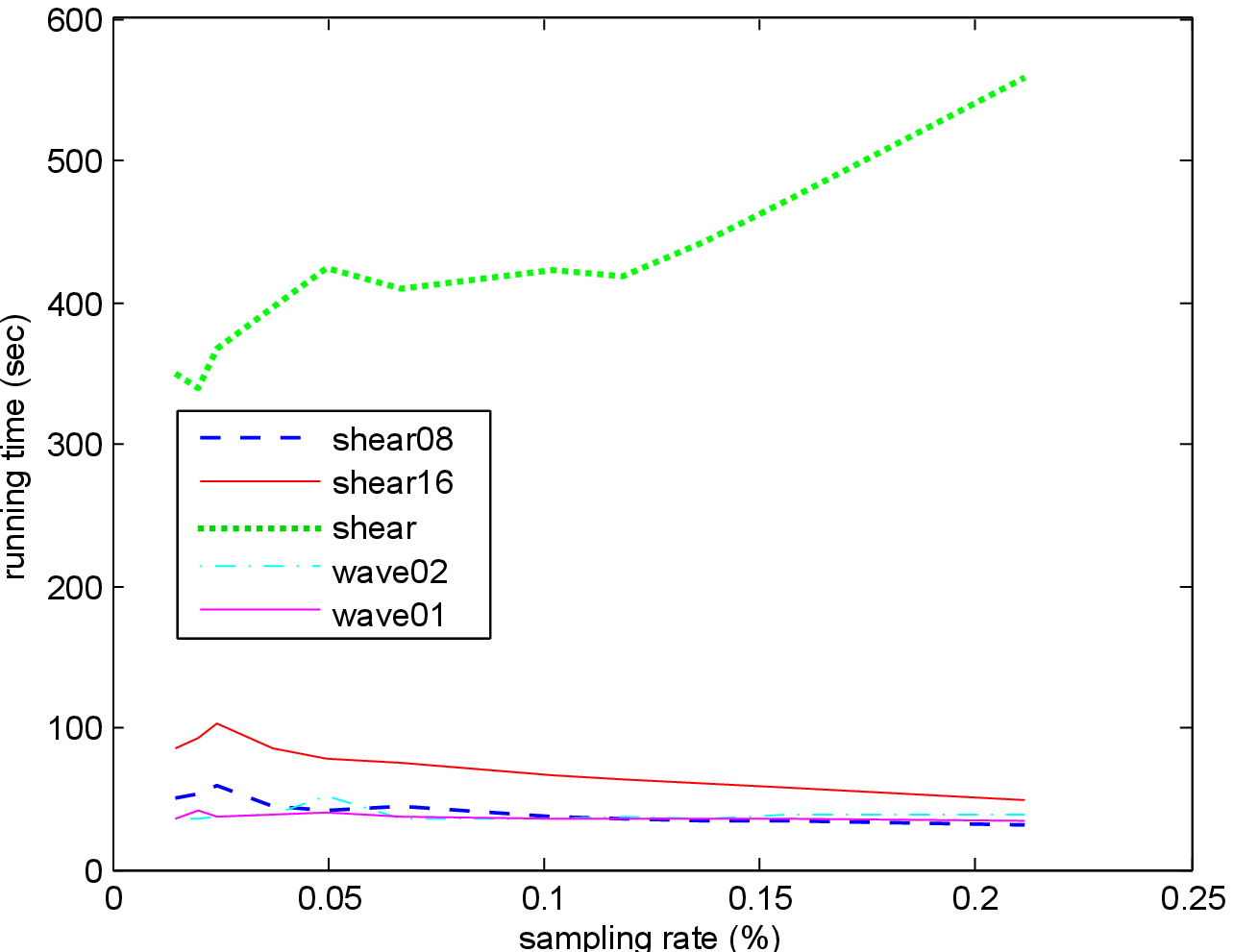}
\put(-290,-13){(e)}
\put(-95,-13){(f)}
\end{center}
\caption{(a) 512 $\times$ 512 test image $A$. (b) 512 $\times$ 512 test image $B$.
(c) PSNR values for various sampling schemes for image $A$.
(d) PSNR values for various sampling schemes for image $B$.
(e) Running time of each of the sampling schemes for image $A$.
(f) Running time of each of the sampling schemes for image $B$.}
\label{fig:experiments}
\end{figure}

A visual comparison of the reconstruction accuracy is provided in Figure \ref{fig:reconstruction}.
As expected from the previous observations, the shearlet based schemes outperform the wavelet
based schemes. \wq{In particular, Figure \ref{fig:reconstruction} (e)--(f) show that curvilinear 
edges are much better preserved by \gk{using an} anisotropic \gk{system for reconstruction}.}

\begin{figure}[h]
\begin{center}
\includegraphics[width=0.3\textwidth]{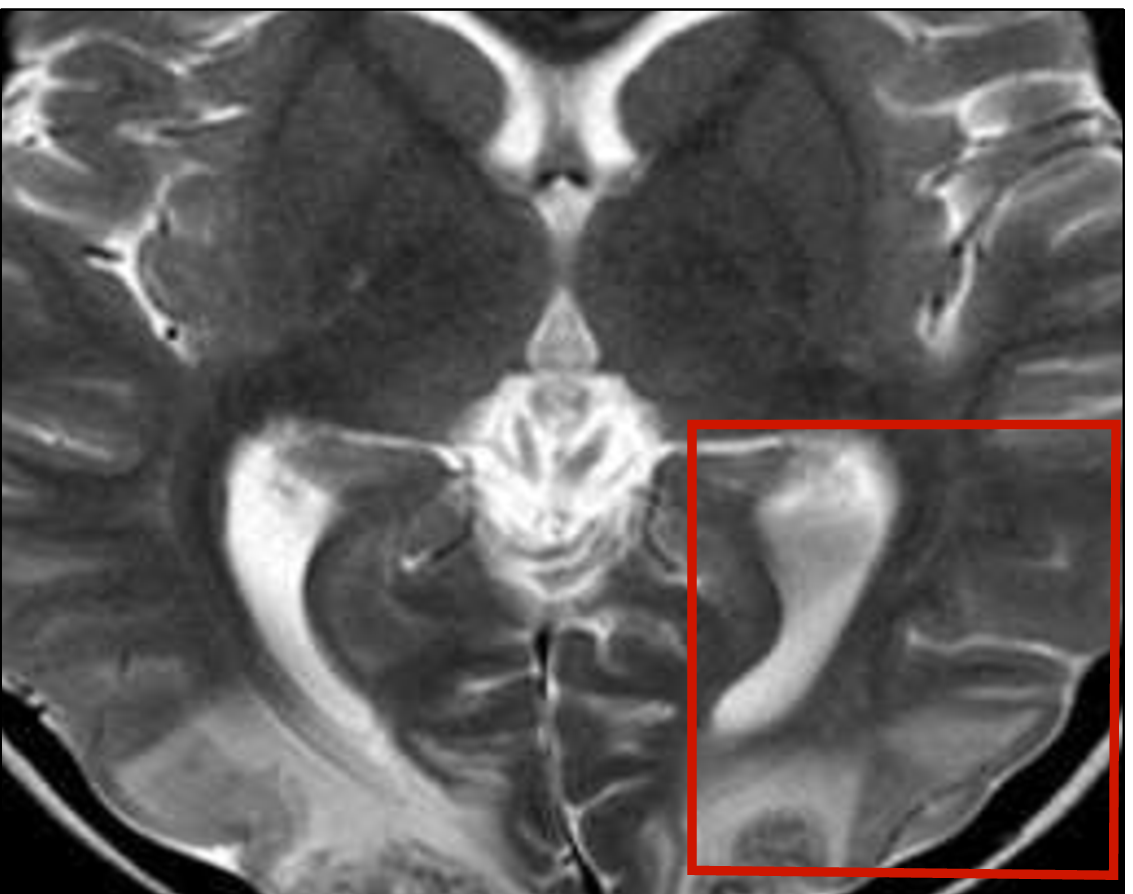}
\includegraphics[width=0.3\textwidth]{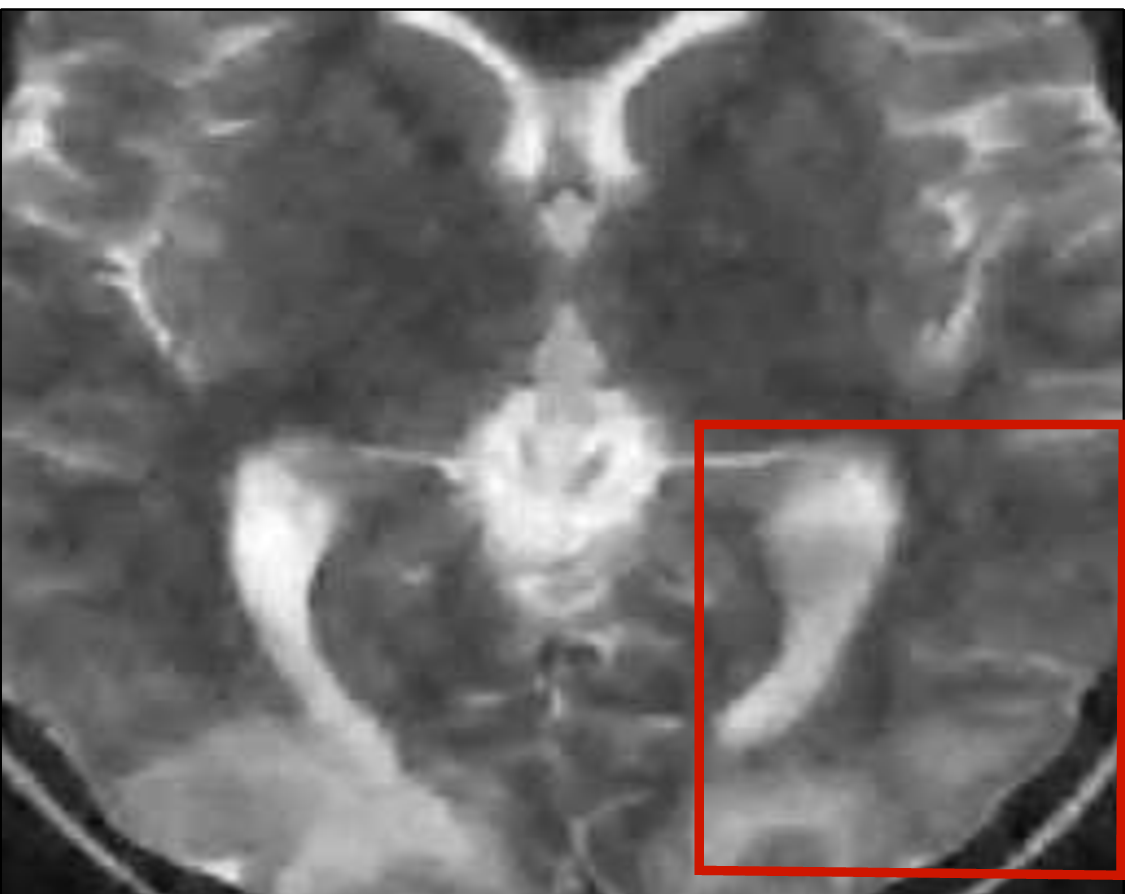}
\includegraphics[width=0.3\textwidth]{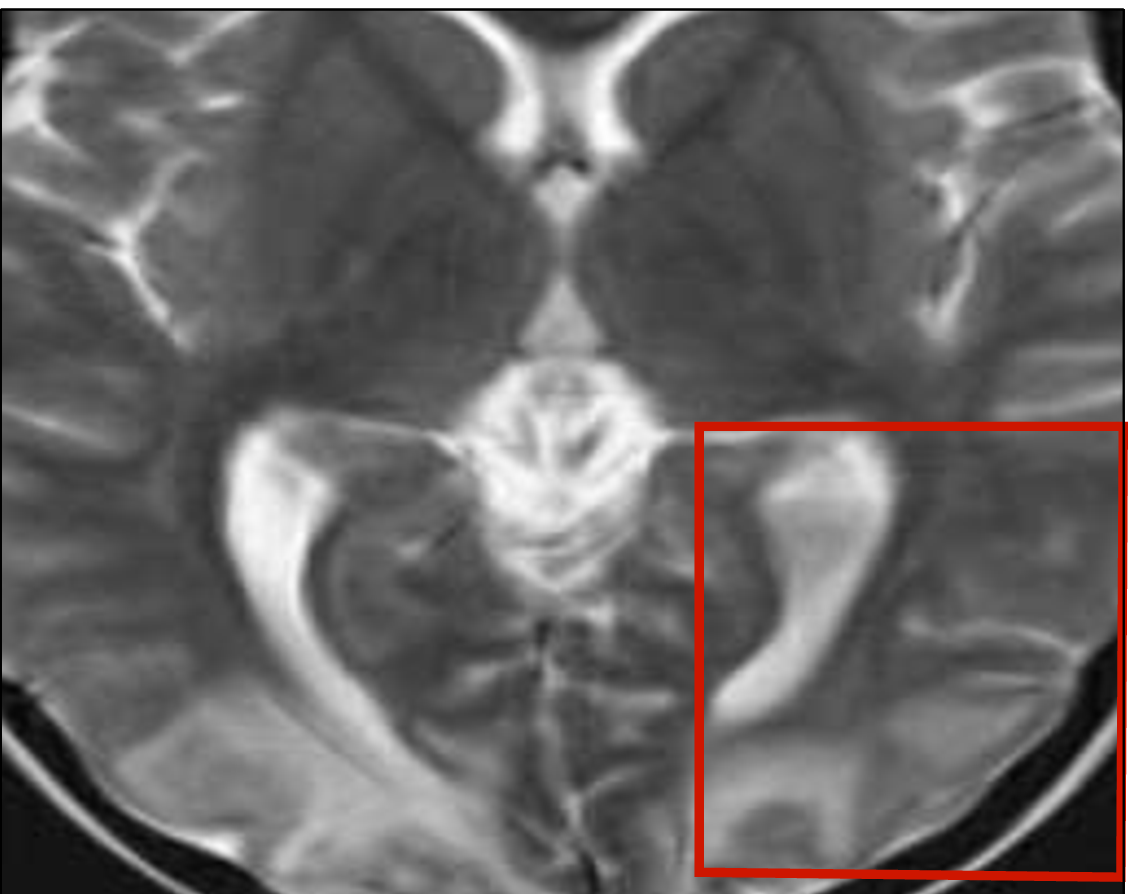}
\put(-365,-13){(a)}
\put(-220,-13){(b)}
\put(-74,-13){(c)}\\[2,5ex]
\includegraphics[width=0.3\textwidth]{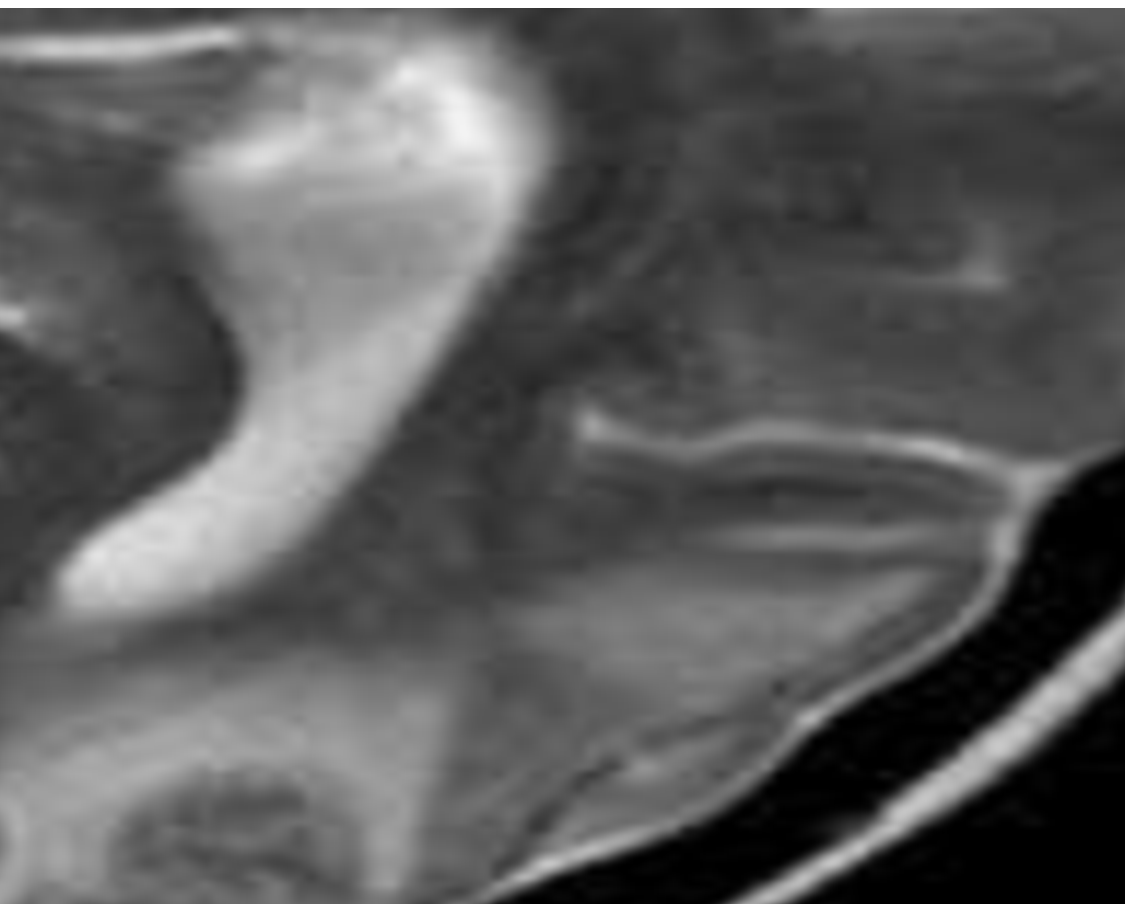}
\includegraphics[width=0.3\textwidth]{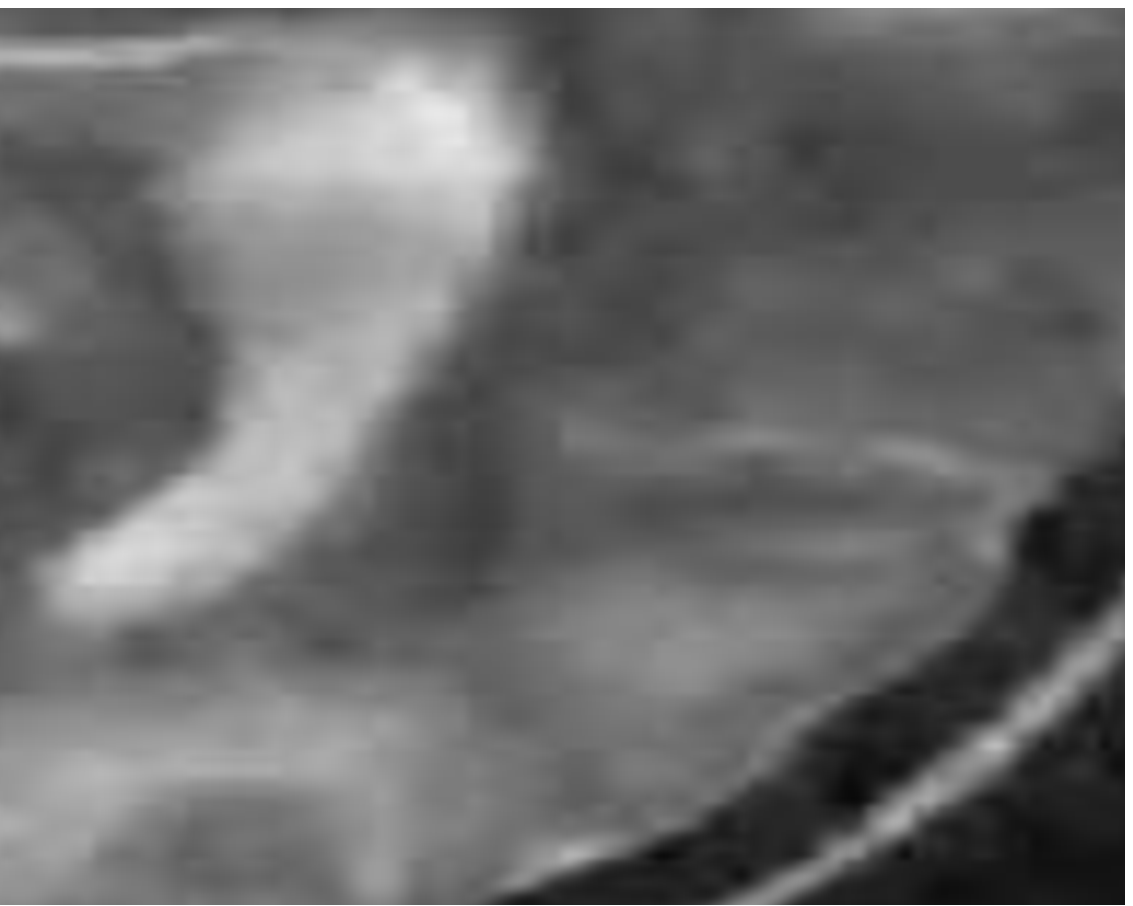}
\includegraphics[width=0.3\textwidth]{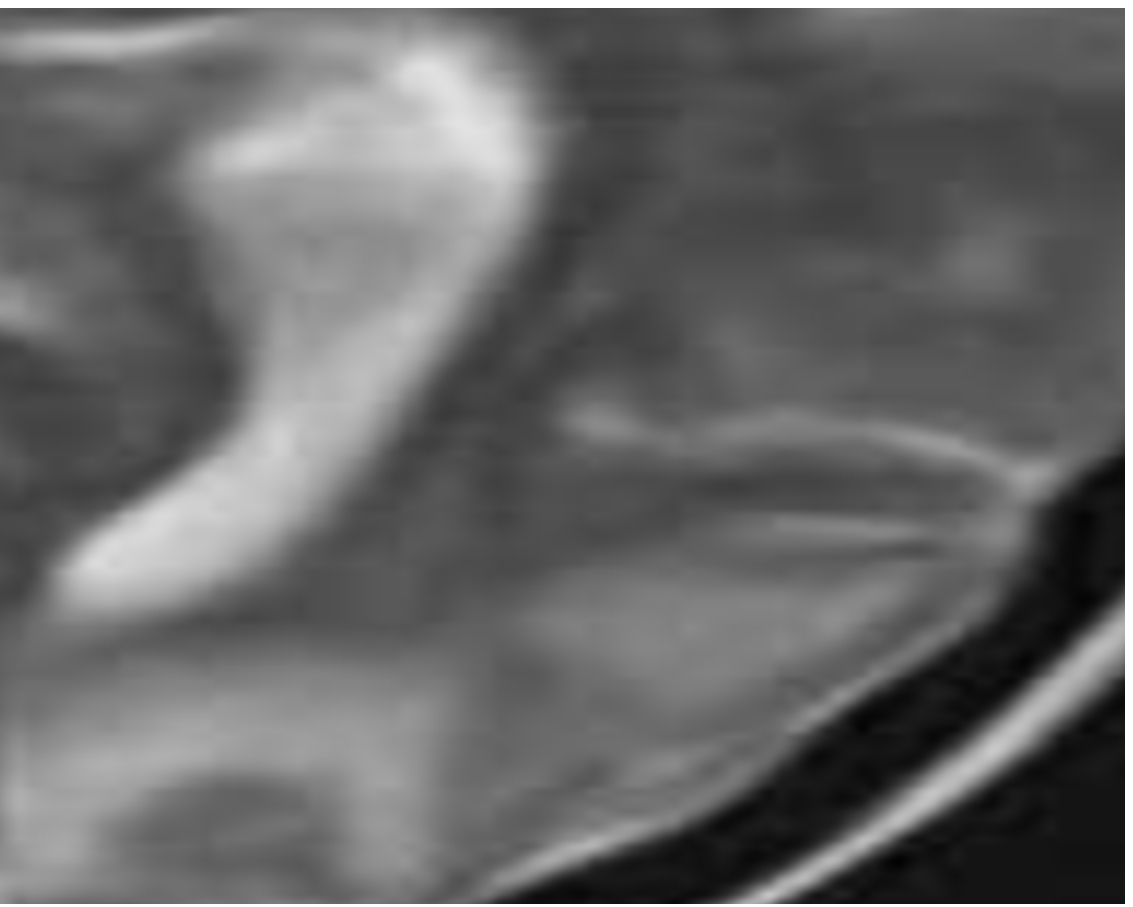}
\put(-365,-13){(d)}
\put(-220,-13){(e)}
\put(-74,-13){(f)}
\end{center}
\caption{Reconstructed images with wavelets and shearlets with our directional sampling using 5\% Fourier coefficients for test image A:
(a) Original image.
(b) {\bf wave01}: Reconstructed image with wavelets \gk{with} 27.02dB.
(c) {\bf shear16}: Reconstructed image with our directional sampling scheme \gk{with} 32.31dB.
(d) Original image (Zoom in).
(e) {\bf wave01}: Zoomed image for (b).
(f) {\bf shear16}: Zoomed image for (c).}
\label{fig:reconstruction}
\end{figure}

\end{document}